\documentclass[12pt, a4paper]{amsart}

\pdfoutput=1
\usepackage{amsmath}
\usepackage{fullpage}
\usepackage[draft=false,backref=none,breaklinks]{hyperref}
\usepackage{ifdraft}
\usepackage[charter]{mathdesign}   
\usepackage{XCharter}
\usepackage{mathtools} 
\usepackage{microtype}
\usepackage{paralist}
\usepackage{thmtools}
\usepackage{tikz}
\usepackage{tikz-cd}
\usepackage{verbatim}
\usepackage[all,cmtip]{xy}

\usetikzlibrary{matrix, arrows}

\ifdraft{\usepackage[textsize=footnotesize, colorinlistoftodos]{todonotes}}
{\usepackage[textsize=footnotesize, colorinlistoftodos, disable]{todonotes}}

\makeatletter
\providecommand\@dotsep{5}
\makeatother
\newcommand{\amsartlistoftodos}{\makeatother \listoftodos\relax}

\definecolor{my-linkcolor}{rgb}{0.75,0,0}
\definecolor{my-citecolor}{rgb}{0,0.5,0}
\definecolor{my-urlcolor}{rgb}{0,0,0.75}
\hypersetup{
    colorlinks, 
    linkcolor={my-linkcolor},
    citecolor={my-citecolor}, 
    urlcolor={my-urlcolor}
}

\newcommand{\ph}{\prescript{p}{}{\operatorname{H}}^0}

\newcommand{\bbC}{\mathbb{C}}

\newcommand{\bbR}{\mathbb{R}}

\newcommand{\bfM}{\mathbf{M}}
\newcommand{\bfN}{\mathbf{N}}

\newcommand{\bfP}{\mathbf{P}}
\newcommand{\bfQ}{\mathbf{Q}}

\newcommand{\calA}{\mathcal{A}}
\newcommand{\calB}{\mathcal{B}}
\newcommand{\calC}{\mathcal{C}}

\newcommand{\calH}{\mathcal{H}}
\newcommand{\calL}{\mathcal{L}}

\newcommand{\calS}{\mathcal{S}}

\newcommand{\scrE}{\mathscr{E}}
\newcommand{\scrF}{\mathscr{F}}

\DeclareMathOperator{\Coind}{Coind}
\DeclareMathOperator{\codim}{codim}

\DeclareMathOperator{\ic}{IC}
\DeclareMathOperator{\id}{id}
\DeclareMathOperator{\Inc}{Inc}
\DeclareMathOperator{\Ind}{Ind}
\DeclareMathOperator{\mor}{Hom}
\DeclareMathOperator{\orient}{or}
\DeclareMathOperator{\Perv}{Perv}

\DeclareMathOperator{\Res}{Res}

\DeclareMathOperator{\Vect}{Vect}

\declaretheorem[parent=section]{theorem}
\declaretheorem[unnumbered, name=Theorem]{theorem*}
\declaretheorem[sibling=theorem]{proposition}
\declaretheorem[unnumbered, name=Proposition]{proposition*}

\declaretheorem[unnumbered, name=Conjecture]{conjecture*}
\declaretheorem[sibling=theorem, style=definition]{definition}
\declaretheorem[unnumbered, style=definition, name=Definition]{definition*}
\declaretheorem[sibling=theorem]{lemma}
\declaretheorem[unnumbered, name=Lemma]{lemma*}

\declaretheorem[unnumbered, style=remark, name=Remark]{remark*}

\declaretheorem[unnumbered, style=remark, name=Assumption]{assumption*}
\declaretheorem[sibling=theorem]{corollary}
\declaretheorem[unnumbered , name=Corollary]{corollary*}
\declaretheorem[sibling=theorem]{question}
\declaretheorem[unnumbered , name=Question]{question*}
\declaretheorem[sibling=theorem]{questions}
\declaretheorem[unnumbered , name=Question]{questions*}
\DeclareMathOperator{\mods}{-mod}
\DeclareMathOperator{\Mods}{-Mod}
\DeclareMathOperator{\Sym}{Sym}

\newcommand{\recollement}[9]{
\begin{tikzcd}[ampersand replacement=\&, column sep=large]
 #1 \arrow{r}{#4}
 \& #2 \arrow[bend right]{l}[swap]{#5} \arrow[bend left]{l}{#6}  \arrow{r}{#7}
 \& #3 \arrow[bend right]{l}[swap]{#8} \arrow[bend left]{l}{#9} 
\end{tikzcd}}

\tikzset{>=to}
\tikzcdset{arrow style=tikz}

\title{Recollement for perverse sheaves on real hyperplane arrangements}
\author{Asilata Bapat}
\date{\today}

\address{Mathematical Sciences Institute, Australian National University, Acton ACT 2601, Australia}
\email{asilata.bapat@anu.edu.au}

\begin{document}
\maketitle
\amsartlistoftodos\

\begin{abstract}
  We consider a hyperplane arrangement in $\bbC^n$ defined over $\bbR$, and the associated natural stratification of $\bbC^n$.
  The category of perverse sheaves smooth with respect to this stratification was described by Kapranov and Schechtman in terms of quiver representations.
  Using work of Weissman, we reinterpret this category as the category of finite-dimensional modules over an explicit algebra.
  We also describe \emph{recollement} (open-closed decomposition) of perverse sheaves in terms of this module category.
  As an application, we identify the modules associated to all intersection cohomology complexes.

  We also compute recollement for $W$-equivariant perverse sheaves for the reflection arrangement of a finite Coxeter group $W$.
  We identify the equivariant intersection cohomology sheaves arising as intermediate extensions of local systems on the open stratum, thereby answering a question of Weissman.
\end{abstract}

\section{Introduction}
Let $X$ be a smooth complex algebraic variety equipped with an algebraic stratification $\calS$.
Let $k$ be a fixed field of coefficients.
The category of perverse sheaves smooth with respect to the stratification $\calS$ is a certain abelian subcategory of the bounded derived category of cohomologically $\calS$-constructible sheaves of $k$-vector spaces on $X$~\cite{bel.ber.del:82}.
This category is typically denoted $\Perv(X,\calS)$.
For simplicity, assume that there is a unique, connected open stratum.
Then $\Perv(X,\calS)$ is a natural generalisation of the category of finite-rank $k$-local systems on the open stratum, which has an algebraic description as the category of representations of the fundamental group of the open stratum.
A general theme is to realise $\Perv(X,\calS)$ as the category of representations of an explicit algebra, perhaps directly extending the above algebraic description of the category of local systems on the open stratum.
There are some general strategies towards such a description~\cite{bel:87,mac.vil:86}, and answers are known in several cases~\cite{gal.gra.mai:85,gel.mac.vil:96,pol:97,bra.gri:99,bra:02,vyb:07,gud.nar:08,str:09,ehr.str:16}.
Our focus is the algebraic description of perverse sheaves on real hyperplane arrangements due to Kapranov--Schechtman~\cite{kap.sch:16}, as well as its equivariant analogue for Coxeter arrangements due to Weissman~\cite{wei:17}.

In general, $\Perv(X,\calS)$ can be reconstructed from the categories of perverse sheaves on the open stratum and on its complement.
This setup is called gluing or \emph{recollement}~\cite{bel.ber.del:82}.
Details about recollements on abelian categories are in \autoref{sec:recollement}.
One of the goals of this paper is to extend the algebraic descriptions of perverse sheaves from~\cite{kap.sch:16} and~\cite{wei:17} to their open/closed recollements, as well as to the subcategories of sheaves supported on closed unions of strata.
We now introduce the setup.

Let $\calH$ be an arrangement of linear hyperplanes in $\bbR^n$.
Let $(\calC,\leq)$ be the partially ordered set of faces of $\calH$, ordered by closure-inclusion.
Let $\calH_{\bbC}$ denote the complexification of $\calH$ in the complex vector space $X = \bbC^n$.
Then $\calH_{\bbC}$ determines a natural stratification $\calS$ on $X$.
More details are in \autoref{subsec:background}.
Kapranov and Schechtman~\cite{kap.sch:16} describe $\Perv(X,\calS)$ as a certain subcategory of the category of representations of the double quiver of the poset $(\calC,\leq)$.
If $\calH$ is the reflection arrangement of a finite Coxeter group $W$, then Weissman~\cite{wei:17} extended the above description to the category $\Perv_W(X,\calS)$, the category of $W$-equivariant perverse sheaves.
Further, Weissman constructed an explicit algebra such that $\Perv_W(X,\calS)$ is equivalent to the finite-dimensional module category over this algebra.
In \autoref{subsec:algebra-r}, we construct a similar explicit algebra $R$ associated to any real arrangement $\calH$.

The main results of the paper are summarised as follows.
\begin{compactenum}[(1)]
\item \autoref{thm:equivalence} gives an equivalence of categories between $\Perv(X,\calS)$ and the category of finite-dimensional $R$-modules.
\item
  \autoref{thm:recollement-equivalence} proves that the recollement of $\Perv(X,\calS)$ into its open and closed pieces is equivalent to a certain recollement for $R$-modules.
\autoref{cor:open-local-systems} identifies the $R$-modules associated to the intersection cohomology sheaves on $X$ arising as intermediate extensions of local systems on the open stratum.
\item \autoref{thm:closed-unions} describes the category of perverse sheaves supported on a closed union of strata of $\calS$ as the category of finite-dimensional modules over a certain quotient of $R$.
  \autoref{cor:main-recollement} identifies the $R$-modules associated to all intersection cohomology sheaves on $X$.
\item \autoref{sec:coxeter-arrangement} discusses the $W$-equivariant setup for the reflection arrangement of a finite Coxeter group $W$.
  \autoref{thm:equivariant-equivalence} is the $W$-equivariant analogue of \autoref{thm:recollement-equivalence}.
  \autoref{cor:equiv-open-ls} is the $W$-equivariant analogue of \autoref{cor:open-local-systems}.
\end{compactenum}

\subsection*{Acknowledgements}
I am indebted to Martin Weissman for helpful conversations and clarifications.
I am grateful to Thomas Gobet for a useful discussion, and to Corey Jones for suggesting the idea of the proof of \autoref{prop:injective-hom}.

\section{Perverse sheaves on real hyperplane arrangements: an alternate description}
The main theorem of~\cite{kap.sch:16} states that the category of perverse sheaves on $\bbC^n$ constructible with respect to $\calS$ is equivalent to a particular subcategory of representations of a quiver.
It will be more transparent for us to describe this subcategory as the module category of an explicit algebra.
The aim of this section is to construct this algebra, and then to translate the data of one of the above admissible quiver representations into the data of a module over this algebra.
We denote this algebra by $R$.

The results of this section should be thought of as parallel to a similar translation from~\cite[Section 4.3]{wei:17}, which focuses on equivariant perverse sheaves on Coxeter arrangements.
The main theorem in this section is \autoref{thm:equivalence}, which obtains an equivalence of categories between $\Perv(X,\calS)$ and finite-dimensional $R$-modules.
\subsection{Real hyperplane arrangements and double quiver representations}\label{subsec:background}\
This subsection recalls some definitions following~\cite{kap.sch:16} and~\cite{wei:17}, as well as the main theorem of~\cite{kap.sch:16}.
For much of the paper, fix $\calH$ to be a hyperplane arrangement in $\bbR^n$, and let $\calH_{\bbC}$ be its complexification in $X = \bbC^n$.
For each $H \in \calH$, also fix a real linear polynomial $f_H$ that cuts out $H$.

For each $x \in \bbR^n$, its real sign vector $\sigma(x) \in \{+,-,0\}^{\calH}$ is defined as follows.
For any $H \in \calH$, the sign $\sigma(x)_H$ at that hyperplane is either $+$, $-$, or $0$, depending on whether $f_H(x)$ is positive, negative, or zero respectively.
We then have an equivalence relation on $\bbR^n$, where $x \sim y$ if $\sigma(x) = \sigma(y)$.
Equivalence classes in $\bbR^n$ with respect to this equivalence relation are called the \emph{faces} of $\calH$, and they form the \emph{face poset} $\calC$.
The partial order on $\calC$ is closure-inclusion: if $C', C \in \calC$, then $C' \leq C$ if $C' \subseteq \overline{C}$.
We will denote by $\calC^i$ the subset of $\calC$ consisting of faces of real codimension $i$.
In particular, elements of $\calC^0$ are sometimes called chambers.

For each $x \in X$, its complex sign vector $\sigma_{\bbC}(x) \in \{\ast, 0\}^{\calH}$ is defined as follows.
For any $H \in \calH_{\bbC}$, the complex sign $\sigma_{\bbC}(x)_H$ at that hyperplane either equals $\ast$ or $0$ depending on whether $f_H(x)$ is non-zero or zero respectively.
If $x, y \in X$, we say that $x \sim_{\bbC} y$ if $\sigma_{\bbC}(x) = \sigma_{\bbC}(y)$.
The set of equivalence classes in $X$ with respect to this equivalence relation forms a stratification of $X$, denoted by $\calS$.
We consider the category $\Perv(X, \calS)$ of perverse sheaves on $X = \bbC^n$ that are smooth with respect to $\calS$.

\begin{definition}
  Let $A$ and $B$ be faces in $\calC^d$, with $d \geq 1$.
  We say that $A$ \emph{opposes} $B$ through $C$ if the following conditions hold.
  \begin{compactenum}[(1)]
  \item The real linear spans of $A$ and $B$ are equal.
  \item There is some face $C\in \calC^{(d-1)}$ such that $C \leq A$, $C \leq B$, and $A$ and $B$ lie on opposite sides of $C$.
    More precisely, for every hyperplane $H \in \calH$ such that $\sigma(C)_H = 0$, we have $\sigma(A)_H = - \sigma(B)_H$.
  \end{compactenum}
\end{definition}

\begin{definition}
  We say that three faces $A$, $B$, and $C$ in $\calC$ are \emph{collinear} if a straight line segment can be drawn in $\bbR^n$ that intersects the faces $A$, $B$, and $C$ in that order.
\end{definition}

\begin{definition}
A double representation $E$ of $\calC$ consists of the following data.
\begin{compactenum}[(1)]
\item A vector space $E_C$ for each $C \in \calC$.
\item Maps $\gamma_{C'C}\colon E_{C'}\to E_C$  and $\delta_{CC'} \colon E_C \to E_{C'}$ for every $C' \leq C$.
\end{compactenum}
For each $C\in \calC$, we require $\gamma_{CC} = \delta_{CC} = \id_{E_C}$.
Additionally, for each $C_1$, $C_2$, $C_3$ in $\calC$ such that $C_1 \leq C_2 \leq C_3$, we require
\[
  \gamma_{C_2 C_3}\gamma_{C_1 C_2} = \gamma_{C_1 C_3}\text {and } \delta_{C_2C_1}\delta_{C_3C_2} = \delta_{C_3C_1}.
\]
\end{definition}
A morphism $V \to W$ of double representations consists of vector space maps $V_C\to W_C$ for each $C\in \calC$, that intertwine with each $\gamma_{C'C}$ and $\delta_{CC'}$.
Double representations of $\calC$ form an abelian category.
\begin{definition}
The category $\calA$ is defined to be the full subcategory of double representations of $\calC$ consisting of objects $E$ satisfying the following three properties.
\begin{compactenum}[(1)]
\item Monotonicity: for every $C' \leq C$, we have $\gamma_{C'C}\delta_{CC'} = \id_{E_C}$.
  As a consequence, we have well-defined maps $\varphi_{AB}\colon E_A \to E_B$ for any $A,B \in \calC$, defined as follows.
  For any $C\in \calC$ satisfying $C \leq A$ and $C \leq B$, set $\varphi_{AB} = \gamma_{CB}\delta_{AC}$.
\item Transitivity: for any $A,B,C\in \calC$ that are collinear (in that order), we have $\varphi_{AC} = \varphi_{BC}\varphi_{AB}$.
\item Invertibility: for any $A,B\in \calC$ that oppose each other, the map $\varphi_{AB}$ is an isomorphism.
\end{compactenum}
\end{definition}

\begin{theorem}[{\cite[Theorem 8.1]{kap.sch:16}}]
  There is an equivalence of categories
  \[
    \bfQ\colon \Perv(X,\calS)\to \calA.
  \]
\end{theorem}
\subsection{The algebra $R$}\label{subsec:algebra-r}
We construct an algebra $R$ to encode the information of the category $\calA$.
The construction parallels the construction of the algebra from~\cite[Theorem 4.3.1]{wei:17} for the Coxeter-equivariant case.
First define $R_0$ to be the algebra generated freely over $\bbC$ by the set $\{e_C \mid C \in \calC\}$, subject to the following relations.
\begin{compactenum}[(R1)]
\item Every generator is idempotent.
  That is, for any face $C$, we have $e_C^2 = e_C$.
\item For any three collinear faces $A$, $B$, and $C$, we have:
  \[
    e_Ae_C = e_Ae_Be_C.
  \]
\item Whenever $A \leq B$, we have
  \[
    e_Ae_B = e_B = e_B e_A.
  \]
\end{compactenum}
Then the idempotent corresponding to the unique smallest face (corresponding to the origin of $\bbR^n$) is the unit of $R_0$.
We define $R$ to be the noncommutative localisation of $R_0$ at the multiplicative subset generated by the following elements, for any two opposing faces $A$ and $B$:
\[
  e_Ae_Be_A + (1-e_A).
\]

If $A$ and $B$ are two opposing faces, let $s_{AB}$ denote the multiplicative inverse of the element $e_Ae_Be_A + (1 - e_A)$, which means that
\[
  (e_Ae_Be_A + (1 - e_A))s_{AB} = s_{AB}(e_Ae_Be_A + (1 - e_A)) = 1.
\]
Since $e_A$ is an idempotent, $(1-e_A)$ is also an idempotent.
Moreover, $e_A(1-e_A) = (1-e_A)e_A = 0$.
Multiplying the above equation on the left and right by $(1-e_A)$ and $e_A$ respectively, we conclude that
\begin{align*}
  (1-e_A)s_{AB} &= s_{AB}(1-e_A) = (1-e_A),\\
  e_Ae_Be_As_{AB} &= s_{AB}e_Ae_Be_A = e_A.
\end{align*}
The first equation also implies that
\begin{equation}
  \label{eq:lr-inverse}
  e_As_{AB} = s_{AB}e_A = e_As_{AB}e_A.
\end{equation}

\subsection{Double quiver representations and $R$-modules}\label{subsec:equiv-functors}
As before, let $\calA$ be the category of monotonic, transitive, and invertible double representations of $\calC$.
Let $R$-mod be the category of finite-dimensional (left) $R$-modules.
We now define functors $\bfM\colon \calA \to R\mods$ and $\bfN\colon R\mods \to \calA$, and show that they are mutually inverse equivalences.

Let $E = (E_C,\gamma_{C'C},\delta_{CC'})$ be an object of $\calA$.
Let $Z$ be the unique smallest face in $\calC$.
Set the underlying vector space of $\bfM(E)$ to be $E_Z$.
For each generator $e_C$ of $R$, set $e_C \colon E_Z\to E_Z$ to be the map $\delta_{CZ}\gamma_{ZC}$.
\begin{proposition}
  Given an object $E$ of $\calA$, the vector space $E_Z$ is an $R$-module via the actions $e_C(v) = \delta_{CZ}\gamma_{ZC}(v)$ for each $C\in \calC$.
\end{proposition}
\begin{proof}
  We check that the actions defined for the generators $e_C$ satisfy the relations of $R$, and also that each element of the localised subset acts invertibly on $E_Z$.
  \begin{compactenum}[(R1)]
  \item Since $\gamma_{ZC}\delta_{CZ} = \id_{E_C}$ for each $C$, the action of every $e_C$ is idempotent.
  \item Let $A$, $B$, and $C$ be three collinear faces.
    Then for each $v \in E_Z$, we have 
    \[
      e_Ae_Be_C(v) = \delta_{AZ}\gamma_{ZA}\delta_{BZ}\gamma_{ZB}\delta_{CZ}\gamma_{ZC}(v).
    \]
    Recall that we have
    \[
      \gamma_{ZA}\delta_{BZ}\gamma_{ZB}\delta_{CZ} = \varphi_{BA}\varphi_{CB} = \varphi_{CA} = \gamma_{ZA}\delta_{CZ},
    \]
    where the second equality is by the transitivity property of $E$.
    All together we have
    \[
      e_Ae_Be_C(v) = \delta_{AZ}\gamma_{ZA}\delta_{CZ}\gamma_{ZC}(v) = e_Ae_C(v).
    \]
  \item Suppose that $A \leq B$.
    Then for each $v \in E_Z$ we have
    \[
      e_Ae_B(v) = \delta_{AZ}\gamma_{ZA}\delta_{BZ}\gamma_{ZB}(v).
    \]
    Since $A \leq B$, we have $\delta_{BZ} = \delta_{AZ}\delta_{BA}$.
  Moreover, $\gamma_{ZA}\delta_{AZ} = \id_{E_A}$.
  So
  \[
    e_Ae_B(v) = \delta_{AZ}\delta_{BA}\gamma_{ZB}(v) = \delta_{BZ}\gamma_{ZB}(v) = e_B(v).
  \]
  A similar computation in  the other direction (using $\gamma_{ZB} = \gamma_{AB}\gamma_{ZA}$) shows that $e_B(v) = e_Be_A(v)$.
\end{compactenum}
Now suppose that $A$ and $B$ oppose each other.
Note that $E_Z = e_AE_Z \oplus (1-e_A)E_Z$, and the element $e_Ae_Be_A + (1-e_A)$ acts by $e_Ae_Be_A$ on $e_AE_Z$ and by $(1-e_A)$ on $(1-e_A)E_Z$.
The latter action is by the identity map, and so obviously an isomorphism.

So checking that the map $e_Ae_Be_A + (1-e_A)$ is invertible on $E_Z$ is equivalent to checking that $e_Ae_Be_A$ is invertible on $e_AE_Z$.
  Since $E_Z$ is finite-dimensional, it is sufficient to check that $e_Ae_Be_A$ is injective on $e_AE_Z$.

  Suppose that for some $e_A(v) \in e_AE_Z$, we have $e_Ae_Be_A(v) = 0$.
  Recall that
  \[
    e_Ae_Be_A(v) = \delta_{AZ}\gamma_{ZA}\delta_{BZ}\gamma_{ZB}\delta_{AZ}\gamma_{ZA}(v) = \delta_{AZ}\varphi_{BA}\varphi_{AB}\gamma_{ZA}(v).
  \]
  Since $A$ and $B$ oppose each other, the maps $\varphi_{AB}$ and $\varphi_{BA}$ are both isomorphisms.
  Since $Z\leq A$, the map $\delta_{AZ}$ is injective.
  This means that if $e_Ae_Be_A(v) = 0$, then $\gamma_{ZA}(v) = 0$.
  Therefore $\delta_{AZ}\gamma_{ZA}(v) = e_A(v) = 0$.
  This calculation implies that $e_Ae_Be_A$ is injective (and hence an isomorphism) on $e_AE_Z$.
  Consequently, $e_Ae_Be_A + (1-e_A)$ is an isomorphism on $E_Z$.
  These checks prove that $E_Z$ acquires the structure of an $R$-module via the actions defined.
\end{proof}

We now define $\bfM$ on morphisms of $\calA$.
Given $\alpha' = (E'_C,\gamma'_{C'C},\delta'_{CC'})$ and a morphism $f\colon \alpha \to \alpha'$, set $\mathbf{M}(f)$ to be the map $f_Z\colon E_Z\to E'_Z$.
For any $C\in \calC$, observe that
\[
  e_Cf_Z = \delta_{CZ}\gamma_{ZC}f_Z = \delta_{CZ}f_C\gamma'_{ZC} = f_Z\delta'_{CZ}\gamma'_{ZC} = f_Ze_C.
\]
Since $f$ commutes with all generators, it is a map of $R$-modules.

Having defined $\bfM\colon \calA \to R\mods$, we now define $\bfN\colon R\mods\to \calA$.
Let $V$ be an $R$-module.
Define $\bfN(V) = (V_C,\gamma'_{C'C},\delta'_{CC'})$ as follows.
For every $C\in \calC$, set $V_C = e_C V$.
Recall that for every $C'\leq C$, we have $e_C = e_{C'}e_C$.
Set $\gamma'_{C'C} \colon e_{C'}V \to e_C V$ to be multiplication by $e_C$.
Set $\delta'_{CC'}\colon e_CV \to e_{C'}V$ to be the natural inclusion map given by $e_Cv \mapsto e_{C'}e_Cv$.
We can think of $\delta'_{CC'}$ as being multiplication by $e_{C'}$.
To ensure that $\bfN(V)$ is an object of $\calA$, we check the following three conditions.
\begin{compactdesc}
\item[Monotonicity]
  Let $C'\leq C$.
  Then for any $e_Cv \in e_CV$, we have
  \[
    \gamma'_{C'C}\delta'_{CC'}(e_Cv) = \gamma'_{C'C}(e_{C'}e_Cv) = e_Ce_{C'}e_Cv = e_Cv.
  \]
  Therefore $\gamma'_{C'C}\delta'_{CC'} = \id_{V_C}$.

\item[Transitivity]
  Let $A$, $B$, and $C$ be collinear faces.
  For any $e_Av \in e_AV$, we have
  \begin{align*}
    \varphi_{BC}\varphi_{AB}(e_AV) &= \gamma'_{ZC}\delta'_{BZ}\gamma'_{ZB}\delta'_{AZ}(e_Av)\\
                                   &=e_Ce_Ze_Be_Ze_Av\\
                                   &=e_Ce_Be_Av\\
                                   &= e_Ce_Av\quad\text{(by the relation in $R$)}\\
                                   & = e_Ce_Ze_Av = \varphi_{AC}v
  \end{align*}

\item[Invertibility]
  Let $A$ and $B$ be two faces that oppose each other.
  Observe that
  \[
    \varphi_{AB}\varphi_{BA} = e_Be_Ae_B,
  \]
  which is invertible on $e_BV$.
  Similarly $\varphi_{BA}\varphi_{AB} = e_Ae_Be_A$, which is invertible on $e_AV$.
  The first relation implies that $\varphi_{AB}$ is surjective and $\varphi_{BA}$ is injective.
  The second relation implies that $\varphi_{BA}$ is surjective and $\varphi_{AB}$ is injective.
  So both $\varphi_{AB}$ and $\varphi_{BA}$ are isomorphisms.
\end{compactdesc}
Therefore $\bfN$ sends $R$-modules to objects in $\calA$.
On a morphism $f \colon V \to W$ of $R$-modules, we set $\bfN(f)_C\colon e_CV \to e_CW$ to be $\bfN(f)_C = e_Cf = f e_C$.

\subsection{Equivalence of categories}
Combining the constructions above with Kapranov--Schechtman's theorem, we obtain the following theorem.
\begin{theorem}\label{thm:equivalence}
  The functors $\mathbf{M}$ and $\bfN$ give mutually inverse equivalences of categories between $R\mods$ and $\calA$.
  Composing with the equivalence $\bfQ \colon \Perv(\bbC^n,\calH) \to \calA$, we conclude that $\Perv(\bbC^n,\calH)$ and $R\mods$ are equivalent via the composition $\bfM\circ \bfQ$.
\end{theorem}

\begin{proof}
  It is easy to see that $\bfM \circ \bfN$ is the identity functor on $R\mods$.
  For the other composition, we show that $\bfN\circ \bfM$ is isomorphic to the identity functor on $\calA$.
  We need isomorphisms $\iota_\alpha \colon \alpha \to \bfN(\bfM(\alpha))$ for each object $\alpha$ of $\calA$, such that if $f\colon \alpha \to \beta$ is any morphism in $\calA$, then the following diagram commutes.
  \begin{equation}
    \label{eq:functoriality}
    \begin{tikzcd}
      \alpha \arrow{r}{\iota_\alpha} \arrow{d}[swap]{f}& \bfN(\bfM(\alpha)) \arrow{d}{\bfN(\bfM(f))}\\
      \beta \arrow{r}{\iota_\beta} & \bfN(\bfM(\beta))
    \end{tikzcd}
  \end{equation}

  Let $\alpha = (E_C,\gamma_{C'C},\delta_{CC'})$ and let $\bfN(\bfM(\alpha)) = (e_C E_Z,\gamma'_{C'C},\delta'_{CC'})$.
  Recall that $\varphi_{CZ}\colon E_C \to E_Z$ can be rewritten as follows:
  \[
    \varphi_{CZ} = \gamma_{ZZ}\delta_{CZ} = \delta_{CZ} = \delta_{CZ}\gamma_{ZC}\delta_{CZ}= e_C \delta_{CZ}.
  \]
  So the image of $\varphi_{CZ}$ lies in $e_CE_Z$.
  We also have the following equalities.
  \begin{align*}
    \varphi_{CZ}\circ\varphi_{ZC} &= \gamma_{ZZ}\delta_{CZ}\gamma_{ZC}\delta_{ZZ} = e_C = \id_{e_C E_Z}\\
    \varphi_{ZC}\circ \varphi_{CZ} &= \gamma_{ZC}\delta_{ZZ}\gamma_{ZZ}\delta_{CZ} = \gamma_{ZC}\delta_{CZ} = \id_{E_Z}.
  \end{align*}
  So $\varphi_{CZ}\colon E_C \to e_CE_Z$ is an isomorphism, with inverse $\varphi_{ZC}$.
  We also check the following for every $C'\leq C$.
  \begin{align*}
      \gamma'_{C'C}\circ \varphi_{C'Z} &= e_C \circ \gamma_{ZZ}\delta_{C'Z} = \delta_{CZ}\gamma_{ZC}\delta_{C'Z}\\
                                       &=\delta_{CZ}\gamma_{C'C}\gamma_{ZC'}\delta_{C'Z} = \delta_{CZ}\gamma_{C'C} = \gamma_{ZZ}\delta_{CZ}\gamma_{C'C}\\
                                       &= \varphi_{CZ}\circ \gamma_{C'C}.\\
    \delta'_{CC'}\circ \varphi_{CZ} &= e_{C'}\gamma_{ZZ}\delta_{CZ} = \delta_{C'Z}\gamma_{ZC'}\gamma_{ZZ}\delta_{CZ} = \delta_{C'Z}\gamma_{ZC'}\delta_{CZ}\\
                                       &=\delta_{C'Z}\gamma_{ZC'}\delta_{C'Z}\delta_{CC'} = \delta_{C'Z}\delta_{CC'} = \gamma_{ZZ}\delta_{C'Z}\delta_{CC'}\\
                                       &= \varphi_{C'Z}\circ \delta_{CC'}.
  \end{align*}
  Therefore the following diagrams commute.
  \begin{center}
    \begin{tikzcd}
      E_{C'} \arrow{r}{\gamma_{C'C}} \arrow{d}[swap]{\varphi_{C'Z}} & E_C \arrow{d}{\varphi_{CZ}}\\
      e_{C'} E_Z \arrow{r}{\gamma'_{C'C}} & e_C E_Z
    \end{tikzcd}
    \quad\quad
    \begin{tikzcd}
      E_{C} \arrow{r}{\delta_{CC'}} \arrow{d}[swap]{\varphi_{CZ}} & E_{C'} \arrow{d}{\varphi_{C'Z}}\\
      e_{C} E_Z \arrow{r}{\delta'_{CC'}} & e_{C'} E_Z
    \end{tikzcd}
  \end{center}
  So the family of maps $\{\varphi_{CZ}\mid C\in \calC\}$ defines an $\calA$-isomorphism from $\alpha$ to $\bfN(\bfM(\alpha))$.
  We set $\iota_\alpha$ to be this isomorphism.

  Let $\beta = (F_C,\gamma_{C'C},\delta_{CC'})$, and let $f \colon \alpha \to \beta$ be any $\calA$-morphism.
  Checking the commutativity of the diagram in~\eqref{eq:functoriality} is equivalent to checking the commutativity of the following diagram for each $C$.
  \begin{center}
    \begin{tikzcd}
      E_C \arrow{r}{(\iota_\alpha)_C} \arrow{d}[swap]{f_C} & e_C E_Z \arrow{d}{e_C f_Z}\\
      F_C \arrow{r}{(\iota_\beta)_C} & e_C F_Z
    \end{tikzcd}
  \end{center}
  This check is completed via the following calculation.
  \begin{align*}
    e_Cf_Z(\iota_\alpha)_C & = \delta_{CZ}\gamma_{ZC}f_Z\gamma_{ZZ}\delta_{CZ}\\
                           &= \delta_{CZ}\gamma_{ZC}\delta_{CZ} f_C\\
                           & = \delta_{CZ}f_Z = \gamma_{ZZ}\delta_{CZ}f_C\\
                           & = (\iota_\beta)f_C.
  \end{align*}
  Therefore $\bfN\circ \bfM$ is isomorphic to the identity functor on $\calA$, and the proof is complete.
\end{proof}

\section{Recollement structures on abelian categories}\label{sec:recollement}
We recall some general definitions, which originally appeared in~\cite[Section 1.4]{bel.ber.del:82}.
Our definitions follow the more recent reference~\cite{fra.pir:04}.
\begin{definition}
  Let $\calA$, $\calA'$, and $\calA''$ be abelian categories.
  Suppose that we have functors $i_*$, $i^*$, $i^!$, $j^*$, $j_!$ and $j_*$ that fit into the following diagram.
  \begin{center}
    \recollement{\calA'}{\calA}{\calA''}{i_*}{i^*}{i^!}{j^*}{j_!}{j_*}
  \end{center}
  Then these functors are said to form a \emph{recollement} if the following conditions hold.
  \begin{compactenum}[(1)]
  \item There are adjunctions $(j_!,j^*)$, $(j^*,j_*)$, $(i^*,i_*)$, and $(i_*,i^!)$.
  \item The unit morphisms $\id_{\calA'} \to i^!i_*$ and $\id_{\calA''}\to j^*j_!$ are isomorphisms.
  \item The counit morphisms $i^*i_* \to \id_{\calA'}$ and $j^*j_*\to \id_{\calA''}$ are isomorphisms.
  \item The functor $i_*$ is an embedding onto the full subcategory of $\calA$ of objects $A$ such that $j^*A = 0$.
    Consequently, the compositions $j^*i_*$, $i^*j_!$, and $i^!j_*$ are zero.
  \end{compactenum}
\end{definition}

\subsection{Recollement of perverse sheaves}
Let $\Perv(X,\calS)$ be the category of perverse sheaves on a space $X$ with respect to a given stratification $\calS$.
Let $U\in \calS$ be an open stratum, and let $V = X \setminus U$.
Let $j \colon U\to X$ and $i \colon V \to X$ be the inclusion maps.
We have a recollement as follows (see, e.g.~\cite{bel.ber.del:82}).
\begin{center}
  \recollement{\Perv(V,\calS|_V)}{\Perv(X,\calS)}{\Perv(U,\calS|_U)}{i_*}{\ph(i^*)}{\ph(i^!)}{j^*}{\ph(j_!)}{\ph(j_*)}
\end{center}

\subsection{Recollement induced by an idempotent on a module category}\label{subsec:ring-recollement}
Let $A$ be a ring and $e$ an idempotent.
We recall the standard recollement induced on $A\mods$ by $e$ in this situation (see, e.g.,~\cite[Section 1]{cli.par.sco:88} or~\cite[Section 4.1]{ste:16}).
Define the functors $\Ind_e$ and $\Coind_e$ from $eAe\mods$ to $A\mods$ as follows:
\[
  \Ind_e(M) = Ae \otimes_{eAe} M, \quad \Coind_e(M) = \mor_{eAe}(eA,M).
\]
Define $\Res_e$ from $A$-mod to $eAe$-mod by
\[
  \Res_e(N) = eN \cong \mor_A(Ae,N) \cong eA\otimes_AN.
\]
Define endofunctors $T_e$ and $N_e$ on $A$-mod as follows:
\[
  T_e(M) = AeM, \quad N_e(M) = \{m \in M \mid eAm = 0\}.
\]
Finally, let $\Inc$ from $A/AeA$-mod to $A$-mod be the functor that upgrades $A/AeA$-modules to $A$-modules via the quotient $A \to A/AeA$.
The following theorem is well-known (see, e.g.~\cite{psa:14}).
\begin{theorem}\label{thm:ring-recollement}
  Let $A$ be a ring and $e\in A$ an idempotent.
  Consider the following functors:
  \begin{align*}
    i_* &= \Inc,\quad i^* = (-)/T_e(-),\quad i^! = N_e,\\
    j^* &= \Res_e, \quad j_! = \Ind_e,\quad j_* = \Coind_e.
  \end{align*}
  These six functors fit into the following recollement diagram.
  \begin{center}
    \recollement{A/AeA\mods}{A\mods}{eAe\mods}{i_*}{i^*}{i^!}{j^*}{j_!}{j_*}
  \end{center}
\end{theorem}

Note that if $M \in eAe\mods$, then there is a natural homomorphism of $A$-modules from $j_!M$ to $j_*M$, described as follows.
Given $a \in A$ and $m \in M$, send $ae\otimes m \in j_!M$ to the element of $j_*(M) =\mor_{eAe}(eA,M)$ given by $eb \mapsto (ebae)m$.
It is easy to check that this is a well-defined map of $A$-modules.
Motivated by the analogous construction in the case of perverse sheaves, we give the following definition.
\begin{definition}\label{def:shriek-star-functor}
  The assignment $M\in eAe\mods$ to the image of the $A$-module homomorphism $j_!(M)\to j_*(M)$ defines a functor from $eAe\mods$ to $A\mods$.
  We call this functor the \emph{intermediate extension functor}, and denote it $j_{!*}$.
\end{definition}

\section{Comparisons between $R$-modules and perverse sheaves}
The aim of this section is to extend the equivalence obtained in \autoref{thm:equivalence} to categories related to $\Perv(X,\calS)$ and specific objects of $\Perv(X,\calS)$.
In particular, we describe an equivalent for $R$-modules of the open/closed recollement of $\Perv(X,\calS)$.
We also describe an equivalent for $R$-modules for the full subcategories of perverse sheaves supported on closed unions of strata.
We use these descriptions to identify intersection cohomology sheaves on $X$ in terms of $R$-modules.
\subsection{An equivalence of recollements}
\begin{lemma}\label{lem:equal-ideals}
  Let $A, F_1, \ldots, F_k, B$ be a sequence of faces in $\calC$ of the same dimension related by successive oppositions.
  That is, each face opposes the next one in order.
  Then the ideal generated by $e_A$ in $R$ contains $e_B$.
\end{lemma}
\begin{proof}
  By induction, it is enough to prove the lemma in the case where $A$ opposes $B$.
  This amounts to showing that the image of $e_B$ is zero in $R/Re_AR$.

  Recall that $s = e_Be_Ae_B + (1-e_B)$ is invertible in $R$.
  So for some $t \in R$, we have $st = ts = 1$.
  Let $\overline{s}$ and $\overline{t}$ denote the images of $s$ and $t$ in $R/Re_AR$.
  We have $\overline{st} = \overline{ts} = 1$ in $R/Re_AR$.
  Note that $\overline{s} = (1 - \overline{e}_B)$.  
  We also know that $e_B$ is idempotent, and so $e_B(1-e_B) = 0$.
  Multiplying the equation $1 = \overline{st}$ on the left by $\overline{e}_B$, we see that $\overline{e}_B = 0$.
\end{proof}

\begin{corollary}
  Let $A, F_1,\ldots, F_k, B$ be a sequence of faces in $\calC$ of the same dimension related by successive oppositions. 
  That is, each face opposes the next one in order.
  Then the categories $e_ARe_A$-mod and $e_BRe_B$-mod are equivalent.
\end{corollary}

\begin{proof}
  By \autoref{thm:ring-recollement}, we have the following two recollements:
  \begin{center}
    \recollement{R/Re_A R\mods}{R\mods}{e_A Re_A\mods}{}{}{}{}{}{}
    \recollement{R/Re_B R\mods}{R\mods}{e_B Re_B\mods}{}{}{}{}{}{}
  \end{center}
  By \autoref{lem:equal-ideals}, we know that $Re_AR = Re_BR$,and so the categories on the left of both recollements are equal.
  Consequently, the three pairs of corresponding functors on the left are also equal in both cases.
  By general properties of recollements, the categories $e_ARe_A$-mod and $e_BRe_B$-mod are both equivalent to the quotient category of $R$-mod by the image of $e_ARe_A$-mod (or equivalently $e_BRe_B$-mod).
\end{proof}

We recall the following definition~\cite[Definition 4.1]{psa.vit:14}.
\begin{definition}
Two recollements $(\calA', \calA, \calA'')$ and $(\calB',\calB,\calB'')$ are said to be \emph{equivalent} if we have equivalences of categories $F'\colon \calA' \to \calB'$, $F\colon \calA \to \calB$ and $F''\colon \calA'' \to \calB''$ that commute with the six recollement functors up to natural equivalence, as shown in \autoref{fig:equivalence}.
\end{definition}
\begin{figure}[h]
  \centering
  \begin{tikzcd}[ampersand replacement=\&, column sep=large, row sep=huge]
    \calA'\arrow{d}[swap]{F'} \arrow{r} {i_*} \& \calA\arrow{d}[swap]{F} \arrow[bend right]{l}[swap]{i^*} \arrow[bend left]{l}{i^!}  \arrow{r}{j^*} \& \calA'' \arrow{d}{F''} \arrow[bend right]{l}[swap]{j_!} \arrow[bend left]{l}{j_*}\\
    \calB' \arrow{r} {i_*} \& \calB\arrow[bend right]{l}[swap]{i^*} \arrow[bend left]{l}{i^!}  \arrow{r}{j^*} \& \calB'' \arrow[bend right]{l}[swap]{j_!} \arrow[bend left]{l}{j_*}     
  \end{tikzcd}
  \caption{An equivalence of two recollements.}\label{fig:equivalence}
\end{figure}
In fact, Lemma 4.2 of~\cite{psa.vit:14} proves that two recollements as above are equivalent if and only if there are equivalences of categories $F\colon \calA \to \calB$ and $F''\colon \calA''\to \calB''$ such that the functors $j^*F$ and $F''j^*$ are naturally equivalent.
\begin{lemma}\label{lem:almost-equivalence}
  Suppose that $(\calA', \calA, \calA'')$ and $(\calB',\calB,\calB'')$ are two recollements.
  Suppose that there are functors $F\colon \calA \to \calB$, $F'\colon \calA'\to \calB'$, and $F''\colon \calA' \to \calB'$ such that the following conditions hold.
  \begin{compactenum}[(1)]
  \item The functors $F$ and $F'$ are equivalences of categories.
  \item The functors $j^*F$ and $F''j^*$ are naturally equivalent.
  \item The functors $i_*F'$ and $Fi_*$ are naturally equivalent.
  \end{compactenum}
  Then the functor $F''$ is full and essentially surjective.
\end{lemma}
\begin{proof}
  Let $(\calA', \calA, \calA'')$ and $(\calB',\calB,\calB'')$ be two recollements, and suppose that the conditions of the lemma hold.
  First we show that $F''$ is full.
  Let $A''_1$ and $A''_2$ be objects of $\calA''$, and let $B''_i = F''(A''_i)$ for $i \in \{1,2\}$.
  Suppose that there is a morphism $f'' \colon B''_1 \to B''_2$ in $\calB''$.
  We will exhibit a lift of $f''$ to a morphism from $A''_1$ to $A''_2$ in $\calA''$.

  Since $j^*j_*$ is isomorphic to $\id_{\calA''}$ on $\calA''$, we can find objects $A_1$ and $A_2$ of $\calA$ such that $j^*(A_i) = A''_i$.
  Let $B_i = F(A_i)$ for $i \in \{1,2\}$.
  Then $j^*B_i = j^*F(A_i)$ is naturally isomorphic to $B_i'' = F''j^*(A_i)$ for each $i$.
  Let $g'' \colon j^*B_1 \to j^*B_2$ denote the map corresponding to $f''$ under the natural isomorphism.
  Since $j^*$ is a full functor on $\calB''$, the morphism $g''$ has a lift $g \colon B_1 \to B_2$ in $\calB$.
  Similarly since $F$ is an equivalence, $g$ has a lift $\widetilde{g} \colon A_1 \to A_2$ in $\calA$.
  Now we can check that $j^*\widetilde{g} \colon A''_1 \to A''_2$ is a morphism in $\calA''$ that lifts $f''\colon B_1''\to B_2''$.
  
  Next we check that $F''$ is essentially surjective.
  Let $B''$ be any object of $\calB''$.
  Recall that the counit morphism $j^*j_*B'' \to B''$ is an isomorphism.
  Since $F$ is essentially surjective, there is some $A \in \calA$ such that $F(A) \cong j_*B''$.
  Let $A'' = j^*A$.
  Then we see that
  \[
    F''(A'') \cong j^*F(A) \cong j^*j_*B'' \cong B''.
  \]
\end{proof}

\begin{proposition}\label{prop:pi1-homomorphism}
  Fix $F \in \calC^0$ and let $e = e_F$.
  Let $U$ be the open stratum in $\calS$, which is also the complex span of $F$.
  Let $\pi_1(U,e)$ denote the fundamental group of $U$ with the basepoint chosen to be some point of $F$.
  There is a ring homomorphism $\iota\colon \bbC[\pi_1(U,e)] \to eRe$.
\end{proposition}
\begin{proof}
  To prove this proposition, we use the Salvetti presentation of the fundamental groupoid $\pi_1(U)$, as described originally in~\cite{sal:87} and also in~\cite[Proposition 9.12]{kap.sch:16}.

  We recall this presentation here.
  There is one object $x_A$ in $\pi_1(U)$ for each $A\in \calC^0$.
  The generating morphisms are $\psi_{AB}\colon x_A\to x_B$ for each ordered pair $(A,B)$ of opposing objects in $\calC^0$.
  Let $\psi^-_{AB}\colon x_A\to x_B$ denote the inverse of the generating morphism $\psi_{BA}\colon B\to A$.
  The relations are as follows.
  Let $F$ be any face in $\calC^2$ (i.e., a face of codimension two).
  As discussed in~\cite[Example 7.9]{kap.sch:16}, let $A$ and $C$ be two elements of $\calC^0$ such that $F < A$, $F < C$, and $C = -A$ is the opposite chamber to $A$.
  Then we can label the chambers around $F$ by $A = B_1, B_2, \ldots, B_{m+1} = C, B_{m+2}, \ldots, B_{2m}$, where any two successive chambers oppose each other, and $B_{2m}$ opposes $A = B_1$.
  For each such instance, we have the \emph{Zifferblatt relation} in $\pi_1(U)$:
  \[
    \psi_{B_m,C}\psi_{B_{m-1},B_m}\cdots\psi_{A,B_2} = \psi_{B_{m+2},C}\psi_{B_{m+3},B_{m+2}}\cdots\psi_{A,B_{2m}}.
  \]
  
  Now if we fix a basepoint in the real chamber $F\in \calC^0$, we can consider the fundamental group $\pi_1(U,e)$.
  The generators and relations of $\pi_1(U,e)$ can be deduced from the generators and relations of $\pi_1(U)$.
  Elements of $\pi_1(U,e)$ consist of all composable words in the letters $\psi_{AB}$ and $\psi^-_{AB}$ that begin and end at $x_F$, with the relations generated by the Zifferblatt relations discussed above.
  This means that locally within any word, any subword consisting of a minimal path between a chamber and its negative can be substituted by the subword formed by the other instance of the minimal path between them.
  Moreover, any instances of $\psi_{AB}\psi^-_{BA}$ or $\psi^-_{AB}\psi_{BA}$ can be cancelled.

  Now define a function $\iota \colon \pi_1(U,e) \to eRe$ as follows.
  Let $w = \psi^{\sigma_k}_{A_k,F}\psi^{\sigma_{k-1}}_{A_{k-1},A_k}\cdots \psi^{\sigma_1}_{F,A_1}$ be any word in $\pi_1(U,e)$.
  We replace any positively-signed letter $\psi_{AB}$ to the ring element $e_Be_A$.
  We replace any negatively-signed letter $\psi^-_{AB}$ by the ring element $e_Bs_{BA}e_B = s_{BA}e_B = e_Bs_{BA}$.
  Then it is easy to see that any word $w$ as above gets sent to an element $\iota(w)$ of $eRe$.
  We check that the relations are satisfied.
  \begin{compactenum}[(1)]
  \item Suppose that $w = w_1\psi_{AB}\psi^-_{BA}w_2$.
    Then we know that $\iota(w_1) = \iota(w_1)e_B$, since $w_1$ ends at $x_B$.
    Therefore
    \[
      \iota(w) = \iota(w_1)e_Be_Ae_Bs_{BA}\iota(w_2) = \iota(w_1)e_B\iota(w_2) = \iota(w_1)\iota(w_2) = \iota(w_1w_2).
    \]
    Similarly if $w = w_1\psi^-_{AB}\psi_{BA}w_2$, then
    \[
      \iota(w) = \iota(w_1)s_{BA}e_Be_Ae_B\iota(w_2) = \iota(w_1)e_B\iota(w_2) = \iota(w_1)\iota(w_2) = \iota(w_1w_2).
    \]
  \item Suppose that $A,C\in \calC^0$ such that $C = -A$, and $F$ is a codimension two face such that $F < A$.
    Number the faces around $F$ starting at $A$ as $B_1,\ldots,B_{2m}$, as in the Zifferblatt relations.
    Let $w = w_1\psi_{B_m,C}\cdots \psi_{A,B_2}w_2$.
    Then
    \[
      \iota(w) = \iota(w_1)e_Ce_{B_m}e_{B_m}e_{B_{m-1}}\cdots e_{B_2}e_{A}\iota(w_2).
    \]
    Since $B_1,B_2,\ldots,B_m$ are successive opposing chambers, any three consecutive ones are collinear.
    So we can telescope the expression $u = e_ce_{B_m}e_{B_m}e_{B_{m-1}}\cdots e_{B_2}e_A$ to $e_Ce_A$.
    Telescoping in reverse for the other path $A = B_{1},B_{2m},\ldots, B_{m+1} = C$, we see that
    \[
      u = e_Ce_{B_{m+2}}e_{B_{m+2}}e_{B_{m+3}}\cdots e_{B_{2m}}e_A = \iota(\psi_{B_{m+2},C}\psi_{B_{m+3},B_{m+2}}\cdots\psi_{B_{2m},A}).
    \]
    This proves the Zifferblatt relation:
    \begin{align*}
      \iota(w) &= \iota(w_1)\iota(\psi_{B_{m+2},C}\psi_{B_{m+3},B_{m+2}}\cdots\psi_{B_{2m},A})\iota(w_2)\\
      &= \iota(w_1\psi_{B_{m+2},C}\psi_{B_{m+3},B_{m+2}}\cdots\psi_{B_{2m},A}w_2).
    \end{align*}
  \end{compactenum}

  So we have showed that $\iota\colon \pi_1(U,e)\to eRe$ maps the group $\pi_1(U,e)$ into the ring $eRe$.
  We extend by linearity to a ring homomorphism $\iota\colon \bbC[\pi_1(U,e)] \to eRe$, and the proposition is proved.
\end{proof}

\begin{lemma}\label{lem:f-double-prime}
  Let $F \colon R\mods \to \Perv(X,\calS)$ denote the (inverse) equivalence from \autoref{thm:equivalence}.
  Fix any $A \in \calC^0$, and let $e = e_A$.
  There is a faithful functor $F_A''\colon eRe\mods \to Perv(U,\calS|_U)$, such that $j^*F$ and $F_A''j^*$ are naturally equivalent.
\end{lemma}
\begin{proof}
  Recall that $\Perv(U,\calS|_U)$ is the category of perverse sheaves on $U$ whose cohomology sheaves are locally constant on $U$.
  So it is equivalent to the abelian category of local systems on $U$, shifted by the complex dimension of $U$.
  In turn, the category of (shifted) local systems on $U$ is equivalent to the category of finite-dimensional representations of its fundamental group.

  As per \autoref{prop:pi1-homomorphism}, there is a ring homomorphism $\iota\colon \bbC[\pi_1(U,e)] \to eRe$, which induces a restriction functor $\iota^*\colon eRe\mods \to \bbC[\pi_1(U,e)]\mods$.
  Define the functor $F_A''$ to be the composition of the equivalence $\bbC[\pi_1(U,e)]\mods \to \Perv(U,\calS|_U)$ with the restriction $eRe\mods \to \bbC[\pi_1(U,e)]\mods$.

  We now show that the functors $j^*F$ and $F_A''j^*$ are naturally equivalent.
  The functor $F$ sends an $R$-module $M$ to a complex of sheaves $\scrE^\bullet(M)$, which is perverse with respect to $\calS$.
  Therefore $j^*(\scrE^\bullet(M))$ is quasi-isomorphic to a (shifted) local system on $U$, namely its left-most homology sheaf.
  The construction from~\cite{kap.sch:16} produces an explicit form for $\scrE^\bullet(M)$ as follows.
  \begin{equation}
    \label{eq:perverse-complex}
    \scrE^\bullet(M) = \left\{ \bigoplus_{\codim(C) = 0}\scrE_C(M)\otimes \orient(C) \to \bigoplus_{\codim(C) = 1}\scrE_C(M)\otimes \orient(C) \to \cdots \right\}
  \end{equation}
  In the above complex, the sheaves $\scrE_C(M)$ are locally constant on a finer stratification of $\bbC^n$ than $\calS$, and $\orient(C)$ is the orientation local system of $C$.
  We refer the reader to~\cite[Section 6]{kap.sch:16} for the full construction, and recall details here as necessary.
  
  So the complex $j^*F(M) = j^*(\scrE^\bullet(M))$ is quasi-isomorphic to the kernel $K$ of the first map in $j^*(\scrE^\bullet(M))$, as shown below.
  \begin{center}
  \begin{tikzcd}
    K\arrow{d}{\sim}\\
    \displaystyle\bigoplus_{\codim(C) = 0}j^*(\scrE_C(M)\otimes\orient(C))\arrow{r}& \displaystyle\bigoplus_{\codim(C) = 1}j^*(\scrE_C(M)\otimes\orient(C)) \arrow{r}& \cdots   \end{tikzcd}
\end{center}

On the other hand, the functor $F''_Aj^*$ first sends the $R$-module $M$ to the $eRe$-module $eM$, and then to the local system defined by the $\pi_1(U,e)$ representation $\iota^*(eM)$.
We now show that $K$ is isomorphic to this local system.

Recall from~\cite{kap.sch:16} that the sheaves $\scrE_C(M)$ are locally constant on cells of the form $iP + Q$, where $P, Q \in \calC$.
The collection of cells of this form refines the stratification $\calS$.
  So we can compute $K$ cell-wise, using the description of the cell-wise stalks of $\scrE_C(M)$ from~\cite{kap.sch:16}.
  If $C,Q\in \calC$, let $C\circ Q$ denote the minimal face $K$ such that $K \geq C$ and $K + \operatorname{Span}_{\bbR}(C) \supset Q$.
  For any cell of the form $iP + Q$, and any face $C \in \calC$, we recall that 
  \begin{equation}
    \label{eq:stalks}
    \scrE_C(M)|_{iP + Q} = \begin{cases}M_{C\circ Q}&\text{if }P \leq C,\\ 0&\text{otherwise}.\end{cases}
  \end{equation}
  Since we are only concerned with $j^*(\scrE^\bullet(M))$, we disregard any cells $iP + Q$ that do not lie in the open stratum $U$.
  These are precisely those $iP + Q$ such that both $P$ and $Q$ are faces of real codimension greater than or equal to $1$.
  In what follows, we only consider cells $iP + Q$ where either $P \in \calC^0$ or $Q \in \calC^0$.

  Let $P\in \calC^0$ be a codimension-zero face, and let $Q\in \calC$ be any other face.
  The description of the stalks in \autoref{eq:stalks} implies that for any $C \in \calC$, the stalk of the sheaf $\scrE_C(M)$ at $iP+Q$ equals $M_P$ if $C = P$, and zero otherwise.
  In particular, we conclude that the stalk of all of $K$ at $iA + Q$ is precisely $M_A = e_AM$.
  Since $K$ is a local system, its stalks at all points of $U$ must be isomorphic to $e_AM$.
  It remains to calculate the monodromy maps.

  Consider any non-trivial element $\ell$ of $\pi_1(U,e)$, which can be written as follows
  \[
    \ell = \psi^{\sigma_{n-1}}_{A_{n-1},A}\psi^{\sigma_{n-1}}_{A_{n-2},A_{n-1}}\cdots \psi^{\sigma_0}_{A,A_1}.
  \]
  Here, $A = A_0, A_1,\ldots, A_{n-1}, A_n = A$ are elements of $\calC^0$ such that $A_k$ and $A_{k+1}$ oppose each other through the codimension-one face $W_k$.
  Then $\ell$ can be represented as a loop in $U$ that begins and ends in the cell $iA + \{0\}$.
  Further, up to homotopy we can choose a representative such that each letter $\psi^{\sigma_k}_{A_k,A_{k+1}}$ is represented by a pair of straight-line segments with the following properties.
  \begin{compactenum}[(1)]
  \item The first segment goes from $iA_k + \{0\}$ to $iW_k + Q_k$ for a certain $Q_k \in \calC^0$.
  \item The second segment goes from $iW_k + Q_k$ to $iA_{k+1} + \{0\}$.
  \item If the sign $\sigma_k$ is positive, then $Q_k = A_k$, otherwise $Q_k = A_{k+1}$.
\end{compactenum}
To calculate the monodromy, it is sufficient to compute the transition maps corresponding to each letter in the word $\ell$.

Consider the letter $\psi^{\sigma_k}_{A_k,A_{k+1}}$.
This corresponds to a transition map from the stalk of $K$ at $iA_k + \{0\}$ to the stalk of $K$ at $iA_{k+1}+ \{0\}$, via the stalk at $iW_k + Q_k$.
As observed earlier, for $\ast \in \{k, k+1\}$, the stalk of $\scrE_C(M)$ at $A_\ast$ equals $M_{A_\ast}$ if $C = A_\ast$, and zero otherwise.
Similarly, observe from \autoref{eq:stalks} that
\[
  \scrE_C(M)|_{iW_k + Q_k} = \begin{cases}M_{A_\ast}&\text{ if }C = A_{\ast}\text{ for }\ast \in \{k,k+1\}\\
    M_{Q_k}&\text{ if }C = W_k\\
    0&\text{ otherwise.}
  \end{cases}
\]
So the stalk of $K$ at $iA_k + \{0\}$ is just $M_{A_k} = e_{A_k}M$, and the stalk of $K$ at $iA_{k+1} + \{0\}$ is just $M_{A_{k+1}} = e_{A_{k+1}}M$.
The stalk of $K$ at $iW_k + Q_k$ is the kernel of the following two-term complex:
\[
  e_{A_k}M \oplus e_{A_{k+1}}M \to M_{Q_k},
\]
where $M_{Q_k} = e_{A_k}$  if $\sigma = +1$ and $M_{Q_k} = e_{A_{k+1}}$ if $\sigma = -1$.
The restriction map $e_{A_k}M \to e_{A_k}M$ is the identity, while the restriction $e_{A_k}M\to e_{A_{k+1}}M$ is the map $e_{A_{k+1}}e_{A_k}$.
Similarly the restriction map $e_{A_{k+1}}M \to e_{A_{k+1}}M$ is the identity, while the restriction $e_{A_{k+1}}M\to e_{A_k}M$ is the map $e_{A_k}e_{A_{k+1}}$.
Note that we have ignored orientations here---the orientation local systems are trivial because each cell is contractible, and the maps on the orientation local systems are canonical.
To compute the kernel, it is only relevant that the restriction maps above have opposite signs, since $A_k$ and $A_{k+1}$ lie on opposite sides of $W_k$.

If $\sigma = +1$, then the stalk of $K$ at $iW_k + Q_k$ is
\[
  \{(x,y) \in e_{A_k}M \oplus e_{A_{k+1}}M \mid x = e_{A_k}e_{A_{k+1}}y\}.
\]
If $\sigma = -1$, then the stalk of $K$ at $iW_k + Q_k$ is
\[
  \{(x,y) \in e_{A_k}M \oplus e_{A_{k+1}}M \mid y = e_{A_{k+1}}e_{A_k}x\} =
  \{(x,y) \in e_{A_k}M \oplus e_{A_{k+1}}M \mid x = e_{A_{k+1}}s_{A_{k+1},A_k}y\}.
\]
So the transition map $\psi^{\sigma_k}_{A_k,A_{k+1}}\colon e_{A_k}M \to e_{A_{k+1}}M$ is
\[
  \psi^{\sigma_k}_{A_k,A_{k+1}} = \begin{cases}x \mapsto e_{A_k}e_{A_{k+1}}x &\text{ if }\sigma_k = +1,\\
    x \mapsto e_{A_{k+1}}s_{A_{k+1},A_k}x,&\text{ if }\sigma_k = -1.
  \end{cases}
\]
Composing the transition maps obtained for each letter $\psi^{\sigma_k}_{A_k,A_{k+1}}$ in $\ell$, we see that this is exactly how we obtained the action of $\ell$ on $\iota^*(eM)$ as constructed in \autoref{prop:pi1-homomorphism}.
This argument shows that $j^*F$ is naturally isomorphic to $F''j^*$.

Finally, we show that $F''$ is faithful.
The functor $\iota^*$ is faithful because it induces the identity functor on the underlying categories of vector spaces.
Since $F''$ is constructed as the composition of an equivalence of categories with $\iota^*$, it is faithful.
\end{proof}

\begin{theorem}
  \label{thm:recollement-equivalence}
  Let $U$ be the open stratum of $\calS$, namely the complement of all of the complex hyperplanes in $\calH$.
  Then $U\cap \bbR^n$ is a disjoint union of all the top-dimensional faces of $\calC$.
  Let $V = \bbC^n \setminus U$.
  Fix any $A \in \calC^0$, and let $e = e_A$.
  Then we have an equivalence of recollements as follows.
  \begin{center}
    \begin{tikzcd}[ampersand replacement=\&]
      R/ReR\mods\arrow{d}[swap]{F_A'} \arrow{r} \& R\mods\arrow{d}[swap]{F} \arrow[shift right=1.5]{l} \arrow[shift left=1.5]{l}  \arrow{r} \& eRe\mods\arrow{d}{F_A''} \arrow[shift right=1.5]{l} \arrow[shift left=1.5]{l}\\
      \Perv(V,\calS|_V) \arrow{r} \& \Perv(X,\calS) \arrow[shift right=1.5]{l} \arrow[shift left=1.5]{l}  \arrow{r} \& \Perv(U,\calS|_U) \arrow[shift right=1.5]{l} \arrow[shift left=1.5]{l}     
    \end{tikzcd}
  \end{center}
\end{theorem}
\begin{proof}
  From the previous lemma, we already have functors $F \colon R\mods \to \Perv(X,\calS)$ and $F_A''\colon eRe\mods \to \Perv(U,\calS|_U)$.
  We now construct an equivalence $F'_A\colon (R/ReR)\mods \to \Perv(V,\calS|_V)$.
  Then we show that these functors give an equivalence of recollements as above.
  
  To define $F'_A$, start with any $M$ in $(R/ReR)\mods$.
  We may think of $M$ as an $R$-module that is annihilated by $e = e_A$.
  Recall from \autoref{lem:equal-ideals} that if $B$ is any other face in $\calC^0$, then $Re_AR = Re_BR$.
  So we may consider $M$ as an $R$-module that is annihilated by all idempotents $\{e_B \mid B \in \calC^0\}$.
  The functor $\bfP\circ \bfN$ from \autoref{thm:equivalence} sends $M$ to the perverse sheaf $\scrE^\bullet(M)$.
  For each $B \in \calC^0$, we have $V_B = e_BV = 0$.
  From this it is easy to check that $\scrE^\bullet(M)$ is term-wise annihilated by $j^*$, and hence is supported on $V$.
  Since the category $\Perv(V,\calS|_V)$ is isomorphic to the full subcategory of $\Perv(X,\calS)$ of objects supported on $V$, this defines the functor $F_A'$.
  By construction, it is clear that $i_*F_A' = Fi_*$.

  To check that $F_A'$ is an equivalence, we define an inverse functor in a similar way.
  If $\scrE^\bullet$ is any perverse sheaf supported on $V$, then the functor $\bfQ$ from \autoref{thm:equivalence} sends $\scrE^\bullet$ to a double quiver representation $(E_C,\gamma_{C'C},\delta_{CC'})$.
  Following~\cite{kap.sch:16}, let $j \colon \bbR^n \to \bbC^n$.
  Then $E_C$ is defined as the global sections of $j^!\scrE$ on $C$.
  Since $\scrE$ is supported on the closed stratum, each $E_C$ is manifestly zero for any $C \in \calC^0$.
  The functor $\bfM$ from \autoref{thm:equivalence} sends the double quiver representation $(E_C,\gamma_{CC'},\delta_{C'C})$ to the $R$-module $E_Z$, where $e_CE_Z = E_C$ for each $C \in \calC$.
  By construction, $e_CE_Z = 0$ for every $C\in \calC^0$, and so we can regard $(\bfM\circ \bfQ)(\scrE)$ as an element of $R/(ReR)\mods$.
  This is the inverse functor to $F_A'$, and so $F_A'$ is an equivalence of categories.
  
  By the construction above and by \autoref{lem:f-double-prime}, we observe that the functors $F$, $F'$, and $F''$ satisfy the conditions of \autoref{lem:almost-equivalence}.
  Therefore $F_A''$ is full and essentially surjective in addition to being faithful, and so it is an equivalence.
  The functors $F$, $F'_A$, and $F''_A$ thus induce an equivalence of recollements as desired.
\end{proof}

As an immediate consequence of the fact that $F''_A$ is an equivalence of categories, we obtain the following corollary.
\begin{corollary}\label{cor:pi1-hom-equivalence}
  The ring homomorphism $\iota$ of \autoref{prop:pi1-homomorphism} induces the pullback functor $\iota^*eRe\mods \to \colon \bbC[\pi_1(U,e)]$ on the respective categories of finite-dimensional modules.
  This functor is an equivalence of categories.
\end{corollary}
Since the finite-dimensional module categories of $\bbC[\pi_1(U,e)]$ and $eRe$ are isomorphic, it is natural to ask whether the map $\iota$ is an isomorphism.
We discuss this question in \autoref{sec:further-questions}.

If $\calL$ is a local system on the open stratum $U$, then the perverse sheaf $j_{!*}(\calL)$ is supported on $X$ and smooth with respect to $\calS$.
It is known as the corresponding \emph{intersection cohomology sheaf}, denoted $\ic(\calL)$.
Using the equivalence of recollements obtained above, we also deduce the following corollary.
\begin{corollary}\label{cor:open-local-systems}
  Let $U$ be the open stratum of $X$, and let $\calL$ be a local system on $U$, corresponding to a representation $L$ of $\pi_1(U,e)$.
  Let $M\in eRe\mods$ be the image of $L$ under the inverse equivalence of $\iota^*$.
  Under the equivalence $\Perv(X,\calS) \cong R\mods$, the intersection cohomology sheaf $\ic(\calL)$ maps to the $R$-module $j_{!*}(M)$, where $j_{!*}$ is the intermediate extension functor from \autoref{def:shriek-star-functor}.
\end{corollary}

\subsection{Perverse sheaves supported on closed unions of strata}\label{subsec:closed-unions}
In this subsection, we extend the previous results to describe the category of perverse sheaves smooth with respect to $\calS$ that are supported on the closure of the union of some of the strata.
As an application, we give a description of all intersection cohomology sheaves on $X$ smooth with respect to $\calS$.

For the remainder of the section, let $Z$ be a closed union of some strata of $\calS$.
That is, $Z = \overline{Z}$.
Note that $Z$ is a union of vector subspaces of $X$ defined over $\bbR$.
Let $\calC_Z$ denote the restriction of the real face poset to $Z$:
\[
  \calC_Z =\{C \in \calC \mid C \subset Z_{\mathbb{R}}\}.
\]
\begin{theorem}\label{thm:closed-unions}
  Let $I_Z$ be the ideal of $R$ generated by the set $\{e_C\mid C\notin \calC_Z\}$.
  Then there is an equivalence of categories $(R/I_Z)\mods \to \Perv(Z,\calS|_Z)$.
\end{theorem}

\begin{proof}
  Since $Z$ is a closed subset of $X$, the category $\Perv(Z,\calS|_Z)$ is isomorphic to the full subcategory of $\Perv(X,\calS)$ consisting of objects supported on $Z$.
  Similarly, we may think of the category $(R/I_Z)\mods$ as the full subcategory of $R\mods$ consisting of objects annihilated by $I_Z$.
  The construction of the functor $(R/I_Z)\mods \to \Perv(Z, \calS|_Z)$ is similar to the construction of the functor $F'_A$ in the proof of \autoref{thm:recollement-equivalence}.

  Let $M \in R\mods$ be an $R$-module that is annihilated by $I_Z$.
  The functor $\bfP \circ \bfN$ from \autoref{thm:equivalence} sends $M$ to a perverse sheaf $\scrE^\bullet(M)$.
  Using the description of the stalks of the component sheaves $\scrE_C(M)$ at cells $iP + Q$ from \autoref{eq:stalks}, we conclude that $\scrE_C(M)|_{iP + Q} = 0$ unless $P \subset C$ and $C \subset Z_{\bbR}$.
  Moreover, even if $C \subset Z_{\bbR}$, the stalk $\scrE_C(M)$ at a cell $iP + Q$ equals zero unless $C\circ Q \subset Z_{\bbR}$.
  Let $Z'_{\bbR} \subset Z_{\bbR}$ be the maximal real vector space that contains $C$.
  Then the cell $C \circ Q$ is a subset of $Z$ if and only if $Q \subset Z'_{\bbR}$.
  This means that both $P$ and $Q$ are contained in a common vector space $Z'_{\bbR} \subset Z_{\bbR}$, which is true if and only if $iP + Q \subset Z$.
  We conclude that the stalks of the component sheaves of $\scrE^\bullet(M)$ are zero at all points outside $Z$.
  In other words, $\scrE^\bullet(M)$ is supported on $Z$, and may be thought of as an object of $\Perv(Z,\calS|_Z)$.
  The inverse equivalence is constructed in a similar manner to the inverse of $F'_A$ in the proof of \autoref{thm:recollement-equivalence}.

\end{proof}

\subsection{Perverse sheaves supported on the closure of a single stratum}\label{subsec:single-stratum-closure}
As a special case of the previous section, let $Z$ be the closure of a single stratum of $\calS$.
Then $\calH^Z_{\bbC}$, the restriction of $\calH_{\bbC}$ to $Z$, is a hyperplane arrangement in $Z$ defined over $\bbR$.
The stratification on $Z$ obtained via $\calH^Z_{\bbC}$ coincides with $\calS|_Z$.
Let $R_Z$ be the algebra defined as in \autoref{subsec:algebra-r} for the arrangement $\calH^Z_{\bbC}$ on $Z$.
Let $I_Z$ be the ideal of $R$ generated by the set $\{e_C \mid C \notin \calC_Z\}$, as above.

By \autoref{thm:closed-unions}, there is an equivalence of categories $(R/I_Z)\mods \to \Perv(Z, \calS|_Z)$.
On the other hand by \autoref{thm:equivalence}, we also have an equivalence $R_Z\mods \to \Perv(Z,\calS|_Z)$.
Our next proposition relates these two.
It is also natural to ask how the algebras $R_Z$ and $R/I_Z$ are themselves related.
We discuss this question in \autoref{sec:further-questions}.
\begin{proposition}\
  \begin{compactenum}[(1)]
  \item There is a surjective ring homomorphism $\rho_Z\colon R/I_Z \to R_Z$, inducing a pullback functor $\rho_Z^*\colon R_Z\mods \to (R/I_Z)\mods$.
  \item The composition 
    \[
      R_Z\mods \xrightarrow{\rho_Z^*} (R/I_Z)\mods \to \Perv(Z,\calS|_Z)
    \]
    coincides with the equivalence of categories constructed in \autoref{thm:equivalence}.
    Consequently, $\rho_Z^*$ is an equivalence of categories.
  \end{compactenum}
\end{proposition}
\begin{proof}
  For clarity, denote the idempotent generators of $R_Z$ by $\{f_C\mid C\in \calC_Z\}$.
  Define a homomorphism $\widetilde{\rho}_Z\colon R\to R_Z$ as follows:
  \[
    e_C \mapsto
    \begin{cases}
      f_C &\text{ if }C \subset T_{\bbR},\\
      0&\text{ otherwise.}
    \end{cases}
  \]
  We check the relations from \autoref{subsec:algebra-r}.
  \begin{compactenum}[(R1)]
  \item The relation (R1) is clear.
  \item Suppose $A$, $B$, and $C$ are three collinear faces.
    If $A\subset T_{\bbR}$ and $C \subset T_{\bbR}$, we must have $B \subset T_{\bbR}$ because $T_{\bbR}$ is a real vector space.
    In this case, $\widetilde{\rho}_Z(e_A) = f_A$, $\widetilde{\rho}_Z(e_B) = f_B$, and $\widetilde{\rho}_Z(e_C) = f_C$ and (R2) clearly holds.
    If one of the three collinear faces lies outside $T_{\bbR}$, then at least one other face must also lie outside $T_{\bbR}$; again because $T_{\bbR}$ is a vector space.
    In this case both sides of (R2) are zero.
  \item Suppose that $A \leq B$.
    If $A\subset T_{\bbR}$ and $B\subset T_{\bbR}$, then again (R3) is clear.
    If $B \not \subset T_{\bbR}$ then both sides of (R3) are zero.
    If $A \not \subset T_{\bbR}$ then it follows that $B \not \subset T_{\bbR}$, and again both sides of (R3) are zero.
  \end{compactenum}
  As for the localisation, suppose that $A$ and $B$ are two faces in $\calC$ that share a wall.
  Since $T_{\bbR}$ is a vector space, $A\subset T_{\bbR}$ if and only if $B\subset T_{\bbR}$.
  If they are both in $T_{\bbR}$, then the element $e_Ae_Be_A + (1-e_A)$ maps under $\widetilde{\rho}_Z$ to $f_Af_Bf_A + (1- f_A)$, which is invertible in $R_Z$.
  If not, then the element $e_Ae_Be_A + (1-e_A)$ maps under $\widetilde{\rho}_Z$ to $1\in R_Z$, which is obviously invertible.
  So the map $\widetilde{\rho}_Z$ as defined on the generators extends to the ring $R$.
  Since $\widetilde{\rho}_Z$ vanishes on $I_Z$ by construction, it factors through $\rho_Z\colon (R/I_Z) \to R_Z$.
  Moreover, the image of $\rho_Z$ contains all generators of $R_Z$ as well as all the adjoined inverses.
  Therefore $\rho_Z$ is surjective.

  Now consider the composition $R_Z\mods \to (R/I_Z)\mods \to \Perv(Z, \calS|_Z)$.
  Suppose that the equivalence from \autoref{thm:equivalence} sends $M \in R_Z\mods$ to the complex $\scrF^\bullet(M)$.
  Let $\scrE^\bullet(M)$ be the complex obtained from the $(R/I_Z)$-module $\rho_Z^\ast(M)$, via the equivalence constructed above.
  Both $\scrF^\bullet(M)$ and $\scrE^\bullet(M)$ are complexes of sheaves that are locally constant on cells of the form $iP + Q$ for $P, Q \in \calC$.
  We compare them stalk-wise to check that they are equal.
  From the expression in \autoref{eq:perverse-complex}, the stalks of $\scrF^\bullet(M)$ are zero by construction on any point of $X$ outside $Z$.
  From the proof above, the stalks of $\scrE^\bullet(M)$ are zero on any point of $X$ outside $Z$.
  Now suppose that $iP + Q$ is a cell where $P \subset Z_{\bbR}$ and $Q \subset Z_{\bbR}$.
  Suppose that $C \in \calC$ with $C \subset Z_{\bbR}$ and $P \leq C$.
  Let $K = C\circ Q$.
  In this case, the stalk of $\scrE_C(M)$ at $iP + Q$ equals $e_K(\rho_Z^\ast(M)) = f_KM$.
  The stalk of $\scrF_C(M)$ at $iP + Q$ also equals $f_KM$.
  Therefore $\scrF^\bullet(M) = \scrE^\bullet(M)$, and the theorem is proved.
\end{proof}

The following proposition says that the composition $R_Z \to (R/I_Z)\mods \to \Perv(Z, \calS|_Z)$ is compatible with open restriction.
We omit the proof, which is similar to arguments from previous results.
\begin{proposition}
  Fix some $A \in \calC_Z$ of maximal dimension in $Z$.
  Let $f = f_A$ be the corresponding idempotent in $R_Z$, and let $e$ be the image of the idempotent $e_A$ in the quotient $R/I_Z$.
  The ring homomorphism $\rho_Z$ restricts to a ring homomorphism $e(R/I_Z)e \to fR_Zf$, and induces an equivalence of categories $fR_Zf\mods \to e(R/I_Z)e\mods$.
  Moreover, the following diagram commutes up to natural isomorphism.
  \begin{center}
    \begin{tikzcd}[row sep=large]
      R_Z\mods\arrow{r} \arrow{d}{j^{-1}} & (R/I_Z)\mods\arrow{r}\arrow{d}{j^{-1}} & \Perv(Z, \calS|_Z)\arrow{d}{j^{-1}}\\
      fR_Zf\mods\arrow{r}\arrow[bend left=45]{u}{j_!}\arrow[bend right=45]{u}[swap]{j_*} &e(R/I_Z)e\mods\arrow{r}\arrow[bend left=45]{u}{j_!}\arrow[bend right=45]{u}[swap]{j_*} & \Perv(Y, \calS_Y)\arrow[bend left=45]{u}{j_!}\arrow[bend right=45]{u}[swap]{j_*}
    \end{tikzcd}
  \end{center}
\end{proposition}
As a consequence, each local system $\calL$ on $Y$ corresponds to an $e(R/I_Z)e$-module.
The next proposition is immediate, and gives a description of intersection cohomology sheaves on $X$ coming from local systems on strata.
\begin{corollary}\label{cor:main-recollement}
  Suppose that $\calL$ is a local system on $Y$, corresponding to the $e(R/I_Z)e$-module $M$.
  Let $j_{!*}(M)$ be the intermediate extension of $M$ to an $(R/I_Z)$-module.
  Under the equivalence $R\mods \to \Perv(X, \calS)$ from \autoref{thm:recollement-equivalence}, the module $j_{!*}(M)$ (thought of as an $R$-module) maps to $\ic(\calL)$.
\end{corollary}

\section{Application to $W$-equivariant perverse sheaves}\label{sec:coxeter-arrangement}
An important example of the kinds of hyperplane arrangements studied above is given by the reflection arrangement of a finite Coxeter group $W$.
In this case we can consider the same setup as in \autoref{subsec:background}.
Additionally, the group $W$ acts on this setup.
We consider the category of $W$-equivariant perverse sheaves on $\calS$, denoted $\Perv_W(X,\calS)$.

Weissman~\cite{wei:17} defines an algebra analogous to the algebra $R$ defined in \autoref{subsec:algebra-r}.
Further, Weissman proves that $\Perv_W(X,\calS)$ is equivalent to the category of finite-dimensional modules over this algebra.
This theorem is the $W$-equivariant analogue of our \autoref{thm:equivalence}.
The aim of this section is to apply the methods from the remainder of the paper to the case of $W$-equivariant perverse sheaves on Coxeter arrangements.

\subsection{Definition of the algebra $R_W$}\label{subsec:equivariant-algebra}
This section recalls some definitions from~\cite{wei:17}.
We refer to~\cite{wei:17} as well as~\cite{bou:02} for details.
In the remainder of this paper we fix $\calH$ to be the real reflection arrangement of a finite Coxeter group $W$.
Otherwise the setup is identical to that in \autoref{subsec:background}.

Further, we fix a chamber $A \in \calC^0$.
Let $S$ be the set of reflections in the walls of $A$.
Then $(W,S)$ forms a finite Coxeter system.
This means that $W$ is generated by the set $S$, modulo the following relations.
\begin{compactenum}[(1)]
\item $s^2 = 1$ for each $s \in S$.
\item $(st)^{m_{st}} = 1$ for each pair $(s, t) \in S$, where $m_{s,t}$ denotes the order of $(st)$ in $W$.
\end{compactenum}
Let $\calC^+ = \{C \in \calC\mid C \leq A\}$.
Note that $\overline{A} = \bigcup_{C\in \calC^+}C$ is a fundamental domain for the action of $W$ on $\bbC^n$.
The set of subsets of $S$, denoted $\Lambda$, is partially ordered by reverse inclusion.
As explained in~\cite[Proposition 2.1.1]{wei:17}, the posets $\calC^+$ and $\Lambda$ are isomorphic.
The isomorphism is given by sending any $I \in \Lambda$ to
\[
  C_I = \{x \in \overline{A}\mid s(x) = x \text{ for all }s \in I, \text{ and }s(x)\neq x \text{ for all }s \notin I\}.
\]
For example, $C_{\emptyset} = A$ and $C_{S} = \{0\}$.

Suppose that $I, J \in \Lambda$ and $w \in W$.
Recall from~\cite[Lemma 4.1.1]{wei:17} that $I$ is said to oppose $J$ through $w$ (written $I\mid_w J$) if the following hold.
\begin{compactitem}[$\bullet$]
\item There is some $K \in \Lambda$ such that $\#I = \#J = \#K - 1$, and $I \cup J \subset K$.
\item $w \in W_K$, and $wJw^{-1} = I$.
\item There are opposite faces $C_1\mid_{C_0} C_2$ such that $C_1 = C_I$, $C_0 = C_K$, and $C_2 = w(C_J)$.
\end{compactitem}

We recall the definition of the algebra from~\cite[Theorem 4.3.1]{wei:17}, which we denote as $R_W$.
It is denoted by $\calA_W$ in the paper above.
\begin{definition}
  Let $R^0_W$ be the algebra generated freely over $\bbC$ by the sets $\{e_I\colon I \in \Lambda\}$ and $\{s \colon s \in S\}$, subject to the following relations.
  \begin{compactenum}[(1)]
  \item For any two $I, J \in \Lambda$, we have $e_Ie_J = e_{I \cap J} = e_Je_I$.
  \item If $s \in I$, then $se_I = e_I s$.
  \item For each $s \in S$, we have $s^2 = 1$.
  \item For any two $s,t\in S$ with $m_{s,t} < \infty$, we have $(st)^{m_{s,t}} = 1$.
  \item Suppose $I$ and $J$ are subsets of $S$ such that $S = I\sqcup J$.
    Let $A \subset I$ and $B \subset J$.
    Let $w_A$ and $w_B$ be the longest elements of the sub-Coxeter systems $(W_A,A)$ and $(W_B,B)$.
    If $w, w_1, w_2 \in W_I$ such that $w = w_2w_1$, and if
    \[
      l(w w_B w_A) = l(w w_B w^{-1}) + l(w_2) + l(w_1) + l(w_A),
    \]
    then
    \[
      e_{A\cup J} \cdot w_1\cdot e_J\cdot w_2\cdot e_{B\cup J} = e_{A\cup J} \cdot w \cdot e_{B\cup J}.
    \]
  \end{compactenum}
  
  Then the algebra $R_W$ is defined to be the localisation of $R^0_W$ at the multiplicative subset generated by the elements of the form
  \[
    \{e_Iw^{-1}e_Jwe_I + (1-e_I)\mid I\mid_wJ\}.
  \]
\end{definition}

\subsection{Equivalence of recollements}
We prove the following equivariant analogue of \autoref{thm:recollement-equivalence}.
\begin{theorem}\label{thm:equivariant-equivalence}
  Let $\calH$ be the real reflection arrangement of a finite Coxeter group $W$, and let $\calH_{\bbC}$ be its complexification.
  Let $\calS$ be the stratification of $X = \bbC^n$ by the complex faces of $\calH_{\bbC}$.
  Let $U \in \calS$ be the open stratum, and let $V$ be its complement in $X$.
  Let $A \in \calC$ be a fixed chamber, so that the closure of $A$ is a fundamental domain for the $W$-action on $X$.
  Let $e \in R_W$ be the idempotent corresponding to $A$ (equivalently, to $\emptyset \in \Lambda$).
  Then we have an equivalence of recollements as follows.
  \begin{center}
    \begin{tikzcd}[ampersand replacement=\&]
      R_W/R_WeR_W\mods\arrow{d}[swap]{F_A'} \arrow{r} \& R_W\mods\arrow{d}[swap]{F} \arrow[shift right=1.5]{l} \arrow[shift left=1.5]{l}  \arrow{r} \& eR_We\mods\arrow{d}{F_A''} \arrow[shift right=1.5]{l} \arrow[shift left=1.5]{l}\\
      \Perv_W(V,\calS|_V) \arrow{r} \& \Perv_W(X,\calS) \arrow[shift right=1.5]{l} \arrow[shift left=1.5]{l}  \arrow{r} \& \Perv_W(U,\calS|_U) \arrow[shift right=1.5]{l} \arrow[shift left=1.5]{l}     
    \end{tikzcd}
  \end{center}
\end{theorem}

\begin{proof}
  First note that $\Perv_W(U,\calS|_U)$ is equivalent to the category of $W$-equivariant local systems on $U$.
  Since the action of $W$ on $U$ is free, this category is equivalent to the category of local systems on $U/W$.
  Recall that the fundamental group of $U/W$ is the Artin braid group $\Gamma_W$ corresponding to $W$.
  We already know from~\cite[Proposition 4.4.2]{wei:17} that there is a ring homomorphism $\iota_W\colon \bbC[\Gamma_W] \to eR_W e$, given by sending $s \in S$ to $ese \in eR_W e$.
  There is a corresponding pullback functor $\iota_W^*\colon eR_We\mods \to \bbC[\Gamma_W]\mods$.

  We now set $F$ to be the equivalence $R_W\mods\to \Perv_W(X,\calS)$ as in~\cite[Theorem 4.3.1]{wei:17}.
  We define $F''_A$ analogously to the definition in \autoref{lem:f-double-prime}.
  The key step of \autoref{lem:f-double-prime} is to show that if $M \in R\mods$, the action of any $\ell \in \pi_1(U,e)$ on $j^*F(M)$ coincides with the action of $\ell$ on the $eRe$-module $eM$ via $\iota^*$.
  We used the fact that any $\ell \in \pi_1(U,e)$ can be written as a product of ``half-monodromies'' around hyperplanes in $\calH$, starting and ending at the basepoint.
  In the equivariant case, the half-monodromies starting at the basepoint are now themselves elements of the fundamental group.
  An analogous calculation goes through in this case, and this is precisely the content of~\cite[Proposition 4.4.1]{wei:17}.
  Finally we define $F'_A\colon R_W/R_WeR_W\mods \to  \Perv_W(V,\calS|_V)$ exactly as in \autoref{thm:recollement-equivalence}.
  The remainder of the proof is analogous to the proof of \autoref{thm:recollement-equivalence}.
\end{proof}

Recall (see, e.g.,~\cite[Section 4.4]{wei:17}) that a $W$-equivariant local system on the open stratum $U$ is just a representation of the braid group $\Gamma_W = \pi_1(U/W)$.
We now have the following analogue of \autoref{cor:open-local-systems}.
\begin{corollary}\label{cor:equiv-open-ls}
  Let $U$ be the open stratum of $X$, and let $\calL$ be an equivariant local system on $U$ corresponding to a representation $L$ of $\Gamma_W$.
  Let $e \in R_W$ be the idempotent corresponding to $\emptyset \in \Lambda$.
  Let $M\in eR_We\mods$ be the object corresponding to $L$ under the equivalence $\Perv(U, \calS|_U) \cong eR_We\mods$.
  Then the intersection cohomology sheaf $IC(\calL)$ corresponds to the $R_W$-module $j_{!*}(M)$, where $j_{!*}$ is the intermediate extension functor on $R_W$-modules as defined in \autoref{def:shriek-star-functor}.
\end{corollary}

Let $Z$ be any $W$-stable and closed union of strata of $\calS$.
Let $\calC^+_Z$ be the subposet of $\calC^+$ defined as $\{C \in \calC^+\mid C \subset Z_{\mathbb{R}}\}$.
Let $\Lambda_Z$ be the image of $\calC^+_Z$ under the identification of posets $\calC^+ \cong \Lambda$.
As in \autoref{subsec:closed-unions}, we obtain a description in terms of $R_W$-modules for the $W$-equivariant perverse sheaves supported on $Z$.
\begin{proposition}
  Let $I_Z$ be the ideal of $R_W$ generated by $\{e_I \mid I \notin \Lambda_Z\}$.
  Then there is an equivalence of categories as follows:
  \[
    (R_W/I_Z)\mods \to \Perv_{W}(Z, \calS|_Z).
  \]
\end{proposition}
\begin{proof}
  Since $Z$ is closed and $W$-stable, the category $\Perv_W(Z, \calS|_Z)$ is isomorphic to the full subcategory of $\Perv(X,\calS)$ consisting of objects supported on $Z$.
  Similarly, the category $(R_W/I_Z)\mods$ is isomorphic to the full subcategory of $R_W\mods$ consisting of objects annihilated by $I_Z$.
  The remainder of the proof is analogous to the proof of \autoref{thm:closed-unions}, using the results of Sections 2.2, 3.3, and Theorem 4.3.1 of~\cite{wei:17} to translate between $W$-equivariant perverse sheaves and $R_W$-modules.
\end{proof}

\section{Further observations and questions}\label{sec:further-questions}

\subsection{Algebras with equivalent finite-dimensional module categories}
Recall that two rings are said to be Morita equivalent if their module categories are equivalent.
In the previous sections, we have discussed several pairs of algebras whose finite-dimensional module categories are equivalent via pullback maps induced from homomorphisms between them.
It is natural to ask whether in these cases this structure induces an isomorphism, or at least a Morita equivalence between these algebras.

If $A$ is an algebra over a field $k$, let $A\Mods$ be the category of all left $A$-modules, including those that are infinite-dimensional over $k$.
Recall that we use $A\mods$ to denote the category of $k$-finite-dimensional left $A$-modules.

\subsubsection{The algebras associated to perverse sheaves on the open stratum}
First consider the $W$-equivariant case for the reflection arrangement of a finite Coxeter group $W$.
Recall from~\cite[Proposition 4.4.2]{wei:17} that there is an algebra homomorphism $\iota \colon \bbC[\Gamma_W]\to eR_We$, where $\Gamma_W$ and $eR_We$ are as defined in \autoref{sec:coxeter-arrangement}.
We see as a corollary of \autoref{thm:equivariant-equivalence} that the pullback functor $\iota^*\colon eR_We\mods \to \bbC[\Gamma_W]\mods$ on the finite-dimensional module categories is an equivalence.
It is not clear from the definition of $\iota$ whether it has any other nice properties.
We prove the following proposition.
\begin{proposition}\label{prop:injective-hom}
  The ring homomorphism $\iota \colon \bbC[\Gamma_W]\to eR_We$ is injective.
\end{proposition}
\begin{proof}
  Suppose that $K = \ker(\iota)$.
  Since the pullback functor $\iota^*\colon eR_We\mods \to \bbC[\Gamma_W]\mods$ is an equivalence, we see that each element of $K$ must act by zero on every finite-dimensional module of $\bbC[\Gamma_W]$.
  To show that $K$ is trivial, it is enough to show that for each non-zero $r \in \bbC[\Gamma_W]$, there is some finite-dimensional module of $\bbC[\Gamma_W]$ on which $r$ does not act by zero.

  Let $r$ be any element of $\bbC[\Gamma_W]$.
  We will find a finite-dimensional representation of $\Gamma_W$ on which $r$ does not act by zero.
  Write $r$ as a finite linear combination
  \[
    r = \sum_{g\in \Gamma_W}c_gg,
  \]
  where each $c_g\in \bbC$  for each $g\in \Gamma_W$, and $c_g = 0$ for all but finitely many $g\in \Gamma_W$.
  It is well-known that $\Gamma_W$ is a linear group~\cite{kra:02,dig:03,coh.wal:02}.
  In particular, there is an embedding $\Gamma_W \xhookrightarrow{\rho} GL(V)$ for some finite-dimensional complex vector space $V$.
  Consider the representations $\Sym^kV$ as $k$ varies over the positive integers.
  For each $k \in \mathbb{N}$ consider $\Sym^k\rho(r)$, which is the element of $GL(\Sym^k(V))$ corresponding to the action of $r$ on $\Sym^kV$.
  If $\rho(r) = \Sym^1\rho(r)$ is nonzero, we are done.

  Otherwise, by choosing a basis $\{v_1,\ldots,v_n\}$ of $V$, we see that each matrix entry of $\rho(r)$ is zero.
  So for each $(i,j)$, we have the following equation:
  \[
    \sum_{g\in \Gamma_W}c_g(\rho(g))_{(i,j)} = 0.
  \]
  The matrix entries of $\Sym^k(\rho(r))$ corresponding to the basis vectors $v_1^k, \ldots, v_n^k$ contain as a subset the $k$th powers of the matrix entries of $\rho(r)$.
  If $\Sym^k(\rho(r))$ were zero for each $k$, then in particular we would have
  \[
    \sum_{g\in \Gamma_W}c_g(\rho(g))_{(i,j)}^k = 0
  \]
  for each positive integer $k$.
  Since all of the above sums are finite, the above equations hold if and only if whenever $c_g\neq 0$, we have $\rho(g)_{(i,j)} = 0$ for all $(i,j)$.
  This means that whenever $c_g\neq 0$, we have $\rho(g) = 0$.
  However, $g$ ranges over $\Gamma_W$, and $\rho \colon \Gamma_W\to GL(V)$ is an embedding.
  So for each $g \in \Gamma_W$, we have $\rho(g) \neq 0$.
  We conclude that $c_g = 0$ for each $g\in \Gamma_W$, which means that $r = 0$.
  This is a contradiction.

  We conclude that for some positive integer $k$, the action of $r$ on the finite-dimensional representation $\Sym^kV$ is nonzero.
  Therefore $K = \ker(\iota)$ is trivial, and the proof is complete.
\end{proof}

We have the following questions.
\begin{questions}\
  \begin{compactenum}[(1)]
  \item Is the homomorphism $\iota \colon \bbC[\Gamma_W]\to eR_We$ an isomorphism?    
  \item Is the pullback functor $\iota^*\colon eR_We\Mods \to \bbC[\Gamma_W]\Mods$ an equivalence of categories?
  \end{compactenum}
\end{questions}

We have a similar situation in the non-equivariant case.
From \autoref{prop:pi1-homomorphism} and \autoref{cor:pi1-hom-equivalence}, we see that the homomorphism $\iota \colon \bbC[\pi_1(U,e)] \to eRe$ induces an equivalence
\[
  \iota^*\colon eRe\mods \to \bbC[\pi_1(U,e)]\mods.
\]
If the group $\pi_1(U,e)$ is linear, then we can use the same argument as in \autoref{prop:injective-hom} to see that the above map $\iota$ is injective.
In general, we have the following questions.
\begin{questions}\
  \begin{compactenum}[(1)]
  \item Is the algebra homomorphism $\iota \colon \bbC[\pi_1(U,e)] \to eRe$ injective? Is it surjective?
  \item Is the pullback functor $\iota^*\colon eRe\Mods \to \bbC[\pi_1(U,e)]\Mods$ an equivalence?
  \end{compactenum}
\end{questions}

\subsubsection{The algebras associated to perverse sheaves on the closure of a stratum}
Recall the setup of \autoref{subsec:single-stratum-closure}.
We have the algebras $R_Z$ and $R/I_Z$, with a surjective ring homomorphism $\rho_Z \colon R/I_Z \to R_Z$ that induces an equivalence $\rho_Z^*\colon (R/I_Z)\mods \to R_Z\mods$.
We again have the following questions.
\begin{questions}\
  \begin{compactenum}[(1)]
  \item Is the algebra homomorphism $\rho_Z\colon (R/I_Z)\to R_Z$ injective?
  \item Is the pullback functor $\rho_Z^*\colon (R/I_Z)\Mods \to R_Z\Mods$ an equivalence?
  \end{compactenum}
\end{questions}

\subsection{$W$-equivariant IC sheaves supported on closures of smaller orbits}
Let $Y$ be the $W$-orbit of a single stratum in the stratification $\calS$ associated to a finite Coxeter arrangement.
Given a $W$-equivariant local system on $Y$, we can apply the intermediate extension functor to obtain a $W$-equivariant perverse sheaf on $\overline{Y}$, which extends by zero to a $W$-equivariant perverse sheaf on $X$.
\begin{question}
  Is there an equivariant analogue of \autoref{cor:main-recollement}?
  In other words, can we describe the restriction functor from $\Perv_W(\overline{Y},\calS|_{\overline{Y}}) \to \Perv_W(Y,\calS_Y)$ in terms of rings as the restriction via an idempotent?
\end{question}

\bibliographystyle{hsiam-doi}
\bibliography{bibliography}

\begin{thebibliography}{10}

\bibitem{bel.ber.del:82}
{\sc A.~A. Be{\u\i}linson, J.~Bernstein, and P.~Deligne}, {\em Faisceaux
  pervers}, in Analysis and topology on singular spaces, {I} ({L}uminy, 1981),
  vol.~100 of Ast\'erisque, Soc. Math. France, Paris, 1982, pp.~5--171.

\bibitem{bel:87}
{\sc A.~A. Be\u{\i}linson}, \href{https://doi.org/10.1007/BFb0078366}{{\em How
  to glue perverse sheaves}}, \href{https://doi.org/10.1007/BFb0078366}{in
  {$K$}-theory, arithmetic and geometry ({M}oscow, 1984--1986)}, vol.~1289 of
  Lecture Notes in Math., Springer, Berlin, 1987, pp.~42--51.

\bibitem{bou:02}
{\sc N.~Bourbaki}, \href{http://dx.doi.org/10.1007/978-3-540-89394-3}{{\em Lie
  groups and {L}ie algebras. {C}hapters 4--6}}, Elements of Mathematics
  (Berlin), Springer-Verlag, Berlin, 2002.
\newblock Translated from the 1968 French original by Andrew Pressley.

\bibitem{bra:02}
{\sc T.~Braden}, \href{https://doi.org/10.4153/CJM-2002-017-6}{{\em Perverse
  sheaves on {G}rassmannians}}, Canad. J. Math., 54 (2002), pp.~493--532.

\bibitem{bra.gri:99}
{\sc T.~Braden and M.~Grinberg},
  \href{https://doi.org/10.1215/S0012-7094-99-09609-6}{{\em Perverse sheaves on
  rank stratifications}}, Duke Math. J., 96 (1999), pp.~317--362.

\bibitem{cli.par.sco:88}
{\sc E.~Cline, B.~Parshall, and L.~Scott}, {\em Finite-dimensional algebras and
  highest weight categories}, J. Reine Angew. Math., 391 (1988), pp.~85--99.

\bibitem{coh.wal:02}
{\sc A.~M. Cohen and D.~B. Wales},
  \href{https://doi.org/10.1007/BF02785852}{{\em Linearity of {A}rtin groups of
  finite type}}, Israel J. Math., 131 (2002), pp.~101--123.

\bibitem{dig:03}
{\sc F.~Digne}, \href{https://doi.org/10.1016/S0021-8693(03)00327-2}{{\em On
  the linearity of {A}rtin braid groups}}, J. Algebra, 268 (2003), pp.~39--57.

\bibitem{ehr.str:16}
{\sc M.~Ehrig and C.~Stroppel},
  \href{https://doi.org/10.1007/s00029-015-0215-9}{{\em Diagrammatic
  description for the categories of perverse sheaves on isotropic
  {G}rassmannians}}, Selecta Math. (N.S.), 22 (2016), pp.~1455--1536.

\bibitem{fra.pir:04}
{\sc V.~Franjou and T.~Pirashvili}, {\em Comparison of abelian categories
  recollements}, Doc. Math., 9 (2004), pp.~41--56.

\bibitem{gal.gra.mai:85}
{\sc A.~Galligo, M.~Granger, and P.~Maisonobe},
  \href{http://www.numdam.org/item?id=AIF_1985__35_1_1_0}{{\em
  {$\mathscr{D}$}-modules et faisceaux pervers dont le support singulier est un
  croisement normal}}, Ann. Inst. Fourier (Grenoble), 35 (1985), pp.~1--48.

\bibitem{gel.mac.vil:96}
{\sc S.~Gelfand, R.~MacPherson, and K.~Vilonen},
  \href{https://doi.org/10.1215/S0012-7094-96-08319-2}{{\em Perverse sheaves
  and quivers}}, Duke Math. J., 83 (1996), pp.~621--643.

\bibitem{gud.nar:08}
{\sc F.~Gudiel-Rodr\'{i}guez and L.~Narv\'{a}ez-Macarro},
  \href{https://doi.org/10.1007/s10468-007-9058-1}{{\em Explicit models for
  perverse sheaves. {II}}}, Algebr. Represent. Theory, 11 (2008), pp.~149--178.

\bibitem{kap.sch:16}
{\sc M.~Kapranov and V.~Schechtman},
  \href{http://dx.doi.org/10.4007/annals.2016.183.2.4}{{\em Perverse sheaves
  over real hyperplane arrangements}}, Ann. of Math. (2), 183 (2016),
  pp.~619--679.

\bibitem{kra:02}
{\sc D.~Krammer}, \href{https://doi.org/10.2307/3062152}{{\em Braid groups are
  linear}}, Ann. of Math. (2), 155 (2002), pp.~131--156.

\bibitem{mac.vil:86}
{\sc R.~MacPherson and K.~Vilonen},
  \href{https://doi.org/10.1007/BF01388812}{{\em Elementary construction of
  perverse sheaves}}, Invent. Math., 84 (1986), pp.~403--435.

\bibitem{pol:97}
{\sc A.~Polishchuk}, \href{https://doi.org/10.4310/MRL.1997.v4.n2.a2}{{\em
  Perverse sheaves on a triangulated space}}, Math. Res. Lett., 4 (1997),
  pp.~191--199.

\bibitem{psa:14}
{\sc C.~Psaroudakis},
  \href{https://doi.org/10.1016/j.jalgebra.2013.09.020}{{\em Homological theory
  of recollements of abelian categories}}, J. Algebra, 398 (2014), pp.~63--110.

\bibitem{psa.vit:14}
{\sc C.~{Psaroudakis} and J.~{Vit\'oria}},
  \href{http://dx.doi.org/10.1007/s10485-013-9323-x}{{\em {Recollements of
  module categories.}}}, {Appl. Categ. Struct.}, 22 (2014), pp.~579--593.

\bibitem{sal:87}
{\sc M.~Salvetti}, \href{http://dx.doi.org/10.1007/BF01391833}{{\em Topology of
  the complement of real hyperplanes in {${\bf C}^N$}}}, Invent. Math., 88
  (1987), pp.~603--618.

\bibitem{ste:16}
{\sc B.~Steinberg}, \href{http://dx.doi.org/10.1007/978-3-319-43932-7}{{\em
  Representation theory of finite monoids}}, Universitext, Springer, Cham,
  2016.

\bibitem{str:09}
{\sc C.~Stroppel}, \href{https://doi.org/10.1112/S0010437X09004035}{{\em
  Parabolic category {$\mathscr{O}$}, perverse sheaves on {G}rassmannians,
  {S}pringer fibres and {K}hovanov homology}}, Compos. Math., 145 (2009),
  pp.~954--992.

\bibitem{vyb:07}
{\sc M.~Vybornov}, \href{https://doi.org/10.1007/s00222-006-0005-2}{{\em
  Perverse sheaves, {K}oszul {IC}-modules, and the quiver for the category
  {$\mathscr{O}$}}}, Invent. Math., 167 (2007), pp.~19--46.

\bibitem{wei:17}
{\sc M.~H. {Weissman}}, \href{http://arxiv.org/abs/1706.07847}{{\em Equivariant
  perverse sheaves on coxeter arrangements and buildings}}, June 2017,
  1706.07847.
\newblock Preprint.

\end{thebibliography}

\end{document}